\documentclass[12pt,reqno]{amsart}
\usepackage{subfiles}
\usepackage{amsmath,amssymb,amsfonts,amscd,latexsym,amsthm,mathrsfs,verbatim, hyperref,tikz-cd}
\usepackage{graphicx}
\usepackage[style=trad-alpha, doi=false, isbn=false]{biblatex}
\usepackage{fullpage}

\renewcommand{\Re}{\operatorname{Re}}

\renewcommand{\div}{\operatorname{div}}

\DeclareMathOperator{\Hom}{Hom}

\DeclareMathOperator{\spec}{Spec}

\DeclareMathOperator{\Bl}{Bl}
\DeclareMathOperator{\rank}{rank}
\DeclareMathOperator{\mult}{mult}
\DeclareMathOperator{\mon}{mon}

\DeclareMathOperator{\coker}{coker}

\DeclareMathOperator{\pr}{pr}

\renewcommand\d{\,\mathrm d}
\def\Proj{{\mathbf{Proj}}}

\DeclareMathOperator{\Pic}{Pic}

\DeclareMathOperator{\Eff}{Eff}
\DeclareMathOperator{\Br}{Br}
\newcommand\ddfrac[2]{\frac{\displaystyle #1}{\displaystyle #2}}

\newtheorem{theorem}{Theorem}
\newtheorem{proposition}[theorem]{Proposition}
\newtheorem{lemma}[theorem]{Lemma}
\newtheorem{corollary}[theorem]{Corollary}

\theoremstyle{definition}
\newtheorem{definition}[theorem]{Definition}
\newtheorem{example}[theorem]{Example}
\newtheorem{remark}[theorem]{Remark}

\newtheorem{assumption}[theorem]{Assumption}
\newtheorem{notation}[theorem]{Notation}

\numberwithin{theorem}{section}
\numberwithin{equation}{section}

\newcommand{\cC}{\mathcal{C}}

\newcommand{\cO}{\mathcal{O}}
\newcommand{\cD}{\mathcal{D}}

\newcommand{\cU}{\mathcal{U}}

\newcommand{\cM}{\mathcal{M}}
\newcommand{\fM}{\mathfrak{M}}

\newcommand{\cX}{\mathcal{X}}
\newcommand{\cY}{\mathcal{Y}}
\newcommand{\cZ}{\mathcal{Z}}
\newcommand{\cG}{\mathcal{G}}

\newcommand{\cK}{\mathcal{K}}
\newcommand{\cHom}{\mathcal{H}\! \mathit{om}}

\newcommand{\bfm}{\mathbf{m}}

\newcommand{\bfs}{\mathbf{s}}
\newcommand{\bfd}{\mathbf{d}}
\newcommand{\bfx}{\mathbf{x}}
\newcommand{\bfe}{\mathbf{e}}
\newcommand{\bfalpha}{\boldsymbol{\alpha}}
\newcommand{\bfbeta}{\boldsymbol{\beta}}

\newcommand{\red}{\textnormal{red}}

\DeclareMathOperator{\Div}{Div}
\DeclareMathOperator{\Volume}{Volume}

\newcommand{\lParen}{(\!(}
\newcommand{\rParen}{)\!)}

\renewcommand{\Im}{\mathrm{Im}\,}

\DeclareFontFamily{U}{wncy}{}
\DeclareFontShape{U}{wncy}{m}{n}{<->wncyr10}{}
\DeclareSymbolFont{mcy}{U}{wncy}{m}{n}
\DeclareMathSymbol{\Sh}{\mathord}{mcy}{"58} 
\DeclareMathSymbol{\B}{\mathord}{mcy}{"42}

\begin{document}
	\title{$\cM$-points of bounded height on toric varieties}
	\author{Boaz Moerman}
	
	
\begin{abstract}
We establish an asymptotic formula for the number of $\cM$-points of bounded height on split toric varieties, for the height induced by any big and nef divisor class. This formula establishes new cases of the extension of Manin's conjecture to $\cM$-points, as introduced by the author. As a special case of our result, we strengthen the results obtained by Pieropan and Schindler on Campana points of bounded height on toric varieties. As another special case, we obtain an asymptotic for the number of weak Campana points of bounded height, which is novel even for projective space. We illustrate this result by giving an asymptotic for the number of points on projective space of bounded height for which the product of coordinates is powerful.

Finally, we derive an asymptotic for the number of rational points in the image of a toric rational map, in the spirit of the Loughran-Smeets conjecture.
\end{abstract}
	\maketitle
	\tableofcontents
	
	\section{Introduction}
In this article, we study the distribution of $\cM$-points on split toric varieties, such as projective spaces. Concretely, we study the set of rational points $(a_1:\dots:a_n)$ for which the tuple of $p$-adic valuations $(v_p(a_1),\dots,v_p(a_n))$ lies in $\fM$ for all prime numbers $p$, for a given set $\fM\subset \mathbb{N}^n$. Many arithmetically interesting sets of rational points are obtained for various choices of the set $\fM$. For instance, the set of points $(a_1:\dots:a_n)$ with $a_1,\dots,a_n$ cubefree, or the set of points for which the product $a_1\dots a_n$ is a squareful number.

In this article, we obtain an asymptotic for the number of $\cM$-points of bounded height on split toric varieties, in the spirit of Manin's conjecture. Since its introduction in the 1980's \cite{FMT89}, Manin's conjecture has inspired many developments in arithmetic geometry and arithmetic statistics, and it has received various refinements \cite{BaMa90, Pey95, BaTs98, LeSeKuSh22}. This conjecture predicts a precise asymptotic for the number of rational points of bounded height on rationally connected varieties.

A few years ago, Manin's conjecture was extended by Pieropan, Smeets, Tanimoto and Várilly-Alvarado to a conjecture on Campana points of bounded height \cite{PSTVA21}. This conjecture has attracted a lot of interest, and various cases of it have been proven \cite{Val12,Xia21,BBKOPW23,PiSc23,PiSc24,CLTT24}.
 In a recent work \cite{Moe25conjecture}, the author further extended Manin's conjecture to the broader setting of $\cM$-points. This conjecture gives a general framework for various results on Campana points and Darmon points which fall outside of the scope of the previous conjectures \cite{Str21, ShSt24,Ito25,Ara25}. In this article, we will prove the latter conjecture for $\cM$-points of bounded height on split toric varieties over $\mathbb{Q}$. Such varieties have proved to be a fertile testing ground for Manin's conjecture and related conjectures, see e.g. \cite{Sal98,Ess07, Hua22, PiSc23,PiSc24, Bon25}.


Let $X$ be a smooth proper split toric variety over $\mathbb{Q}$, with torus invariant prime divisors $D_1,\dots,D_n$. Let $\fM\subset \mathbb{N}^n$ be a subset containing the origin and some positive multiple of each basis element $\mathbf{e}_1,\dots,\mathbf{e}_n$, let $(X,M)$ be the corresponding proper toric pair with toric integral model $(\cX,\cM)$ as in Definition \ref{def: toric pair}.
For an integer $S\geq 1$, we write
$$(\cX,\cM)(\mathbb{Z}[\tfrac{1}{S}])=\{P\in X(\mathbb{Q})\mid \mult_p(P)\in\fM\text{ for all prime numbers } p\nmid S\}$$
for the set of $\cM$-points over $\mathbb{Z}[\tfrac{1}{S}]$ on $(\cX,\cM)$.
Here the multiplicity $\mult_p(P)$ is given as  $\mult_p(P)=(v_p(a_1):\dots:v_p(a_n))$, where $P=(a_1:\dots:a_n)$ are Cox coordinates for the point as defined in Section \ref{section: Cox coordinates}.

For a big and nef $\mathbb{Q}$-divisor on $X$, let $$H_L\colon X(\mathbb{Q})\rightarrow \mathbb{R}_{>0}$$ be the corresponding height as defined in Section \ref{section: heights on toric variety}. For an integer $B$, we write $$N_{(X,M),L,S}(B)=\#\{P\in (\cX,\cM)(\mathbb{Z}[\tfrac{1}{S}])\mid H_L(P)\leq B \}.$$
We give a precise asymptotic for this counting function as $B\rightarrow \infty$, establishing new cases of \cite[Conjecture 1.4]{Moe25conjecture}.
\begin{theorem}\label{theorem: asymptotics proper toric}
	Let $(X,M)$ be a smooth proper toric pair over $\mathbb{Q}$ with toric integral model $(\cX,\cM)$ over $\mathbb{Z}$ and let $L\in \Pic(X)$ be a big and nef divisor class. Then there exists $\theta>0$ and a polynomial $Q$ of degree $b(\mathbb{Q},(X,M),L)-1$ such that
	$$N_{(X,M),L,S}(B)=B^{a((X,M),L)}\left(Q(\log B)+O\left(B^{-\theta}\right)\right).$$
	Here $a((X,M),L)$ and $b(\mathbb{Q},(X,M),L)$ are the Fujita invariant and the $b$-invariant of $(X,M)$ with respect to $L$ as defined in \cite{Moe25conjecture}, which we recall in Definition \ref{def: a and b invariant}.
	Furthermore, if $L$ is adjoint rigid with respect to $(X,M)$, then the leading coefficient of $Q$ is the leading constant from \cite[Conjecture 1.4]{Moe25conjecture}, and it is explicitly given in Theorem \ref{theorem: full asymptotics} as well as in Theorem \ref{theorem: geometric interpretation constant}.
\end{theorem}

	Here the condition that $L$ is adjoint rigid means that the class $a((X,M),L)\pr_M^* L+K_{(X,M)}$ is represented by a unique $\mathbb{Q}$-divisor on $(X,M)$, see Definition \ref{def: toric adjoint rigid}.
	This theorem is a special case of Theorem \ref{theorem: full asymptotics}, which only requires the pair to be quasi-proper as in Definition \ref{def: quasi-proper pair}, rather than proper. To prove Theorem \ref{theorem: asymptotics proper toric} for divisors which are not adjoint rigid with respect to $(X,M)$, we need to consider quasi-proper pairs to show that $Q$ has degree equal to, rather than at most, $b(\mathbb{Q},(X,M),L)-1$.
	
	Theorem \ref{theorem: asymptotics proper toric} implies \cite[Conjecture 1.1]{PSTVA21} for split toric varieties over $\mathbb{Q}$, including the conjecture for the leading constant if $L$ is adjoint rigid. This generalizes the results by Pieropan and Schindler \cite[Theorem 1.2]{PiSc23} to heights corresponding to divisors different from the log-anticanonical divisors, and improves on their error term.
	The theorem is proved using the universal torsor method, as developed by Salberger \cite{Sal98}.

The proof proceeds along the lines of de la Bretèche's proof \cite{Bre01Toric} of Manin's conjecture for split toric varieties with the anticanonical height, together with Salberger's computation of the leading constant \cite[Section 11]{Sal98}. Using the universal torsor, we identify the set of $\cM$-points on $(\cX,\cM)$ with a subset of $\mathbb{Z}^n$, and we study the counting function via a related multivariate Dirichlet series. We then derive the asymptotic by applying two Tauberian theorems by de la Bretèche \cite{Bre01Sum} to this series.

Our proof generalizes de la Bretèche's proof \cite{Bre01Toric} of Manin's conjecture for split toric varieties in two ways: it generalizes generalizing his results to $\cM$-points, and it furthermore works with other heights than the anticanonical height.
	
	Theorem \ref{theorem: asymptotics proper toric} is of interest even in the classical setting of rational points, as it improves on the original proof of Manin's conjecture for toric varieties by Batyrev and Tschinkel \cite[Corollary 1.5]{BaTs96} by providing a good control of the error term. 
	In the classical setting of rational points, Theorem \ref{theorem: asymptotics proper toric} and its proof are similar to \cite[Theorem 1]{Ess07}, where heights coming from other metrizations are considered.

	We now give a few examples to illustrate Theorem \ref{theorem: asymptotics proper toric}. In these examples, we will consider weak Campana points.
	\begin{example}
		Theorem \ref{theorem: asymptotics proper toric} implies that
		$$\#\left\{(x:y:z)\in \mathbb{P}^{2}(\mathbb{Q})\,\middle\vert
		\begin{aligned} &\,x,y,z\in \mathbb{Z}\setminus\{0\},\, \gcd(x,y,z)=1, \\
			&\, xyz \text{ is squareful},\, \max(|x|,|y|, |z|)\leq B
		\end{aligned}
		\, \right\}
		= B^{3/2} (Q(\log B)+O(B^{-\theta}))$$
		as $B\rightarrow \infty$, where $\theta>0$ is a constant and $Q$ is a cubic polynomial with leading coefficient
		$$\prod_{p \text{ prime}} \left(1-p^{-1}\right)^6 \left(\frac{1-p^{-3/2}}{\left(1-p^{-1/2}\right)^3} -3p^{-1/2}\right)\approx 0.862.$$
		See Example \ref{example: weak Campana points projective plane} for the derivation of this result.
	\end{example}
	
	\begin{example} \label{example: introduction weak Campana projective space points with equal weights}
		More generally, Theorem \ref{theorem: asymptotics proper toric} implies that for any positive integers $m$ and $n$,
		\begin{multline*}
			\#\left\{(x_1:\dots: x_n)\in \mathbb{P}^{n-1}(\mathbb{Q})\middle\vert
			\begin{aligned} &\,x_1,\dots, x_n\in \mathbb{Z}\setminus \{0\},\, \gcd(x_1,\dots,x_n)=1, \\
				&\, \prod_{i=1}^n x_i \text{ is }m\text{-full},\, \max(|x_1|,\dots, |x_n|)\leq B
			\end{aligned}
			\, \right\}=\\ B^{n/m}(Q(\log B)+O(B^{-\theta}))
		\end{multline*}
		as $B\rightarrow \infty$, where $\theta>0$ is a constant and $Q$ is a polynomial of degree
		$$\binom{m+n-1}{n-1}-\binom{m-1}{n-1}-n.$$
		We will derive this result in Example \ref{example: weak Campana projective space points with equal weights}.
	\end{example}
	\begin{remark}
	This result is related to Streeters result on powerful values of norm forms \cite[Theorem 1.1]{Str21}, as the Campana pair considered in that article can be regarded as a Galois twist of $(\mathbb{P}^n_{\mathbb{Q}}, \tfrac{1}{m} \sum_{i=1}^n D_i)$. We discuss his result in further detail in \cite[\S 9.5]{Moe25conjecture}.	
	\end{remark}

	More generally, Theorem \ref{theorem: asymptotics proper toric} gives an asymptotic for the number of weak Campana points of bounded height on a split toric variety. See \cite[Theorem 1.7]{Moe25conjecture} for a direct description of the Fujita invariant and the $b$-invariant in this setting.

As another consequence of Theorem \ref{theorem: asymptotics proper toric}, we also find an asymptotic formula for the number of rational points of bounded height in the image of a toric rational map, i.e. a rational map of toric varieties restricting to a homomorphism on the dense tori.
\begin{corollary} \label{corollary: image pair}
	Let $f\colon Y\dashrightarrow X$ be a dominant toric rational map of smooth proper split toric varieties, let $(Y,M)$ be a proper toric pair with toric integral model $(\cY,\cM)$, and let $L\in \Pic(X)$ be a big and nef divisor class. Then for all integers $B$, we have that
	$$\#\{P\in f((\cY,\cM)(\mathbb{Z}))\mid H_L(P)\leq B\}=\Delta^{-1} N_{(X,f_*M),L,1}(B).$$
	Here $(X,f_*M)$ is the proper toric pair as in Definition \ref{def: image pair}, and $\Delta$ is the order of $(N/\overline{f}N')\otimes \mathbb{Z}/2\mathbb{Z}$, where $\overline{f}\colon N'\rightarrow N$ is the map of lattices corresponding to $f$.
\end{corollary}
This corollary is a direct consequence of Lemma \ref{lemma: image pair} together with the fact that the height of a rational point $(a_1,\dots:a_n)$ does not depend on the signs of $a_1,\dots,a_n$.
Theorem \ref{theorem: asymptotics proper toric} and Corollary \ref{corollary: image pair} together give an explicit asymptotic for the number of rational points on $X$ of bounded height which lie in the image of an $\cM$-point. In particular, this gives a formula for the number of rational points of bounded height in the image of $f$. This is in the spirit of the Loughran--Smeets conjecture and its extensions, as studied for instance in \cite{LRS24,BLS25}. It does not fall within the framework of the conjecture however, as the generic fiber of a toric morphism is usually not geometrically irreducible.

\begin{example} \label{example: counting images}
	Let $f\colon \mathbb{P}^2\dashrightarrow \mathbb{P}^2$ be the rational map given by $(x:y:z)\mapsto (x^2:y^2:xz)$. Then the number of points in $f(\mathbb{P}^2(\mathbb{Q}))\subset \mathbb{P}^2(\mathbb{Q})$ of naive height at most $B$ tends to $B^2(Q(\log B)+O(B^{-\theta}))$, where $\theta>0$ and $Q$ is a linear polynomial with leading coefficient $$\frac{1}{2}\prod_{p \text{ prime}}(1-2p^{-2}+2p^{-3}-p^{-4})\approx 0.245.$$ We will derive this asymptotic in Example \ref{example: counting images worked out}.
\end{example}

\subsection{Structure of the paper}
In Section \ref{section: geometry toric pairs}, we introduce toric pairs and study their geometry. We will first introduce Cox coordinates on toric varieties, which generalize the homogeneous coordinates on projective space. After that, we compute the Picard group of a toric pair. We finish the section by introducing quasi-proper pairs and the universal torsor of a toric pair.

In Section \ref{section: M points bounded height}, we recall the toric height used in this paper, and use it to state Theorem \ref{theorem: full asymptotics}. This theorem gives an asymptotic formula for the counting function $N_{(X,M),L,S}$ for any quasi-proper toric pair $(X,M)$, and recovers Theorem \ref{theorem: asymptotics proper toric} as a special case. Afterwards, we illustrate the theorem in a few examples and give a geometric interpretation of the leading constant.

Finally, we prove the Theorem \ref{theorem: full asymptotics} in two parts. In the first part we use \cite[Théorème 1]{Bre01Sum} to show the asymptotic $N_{(X,M),L,S}(B)=B^{a((X,M),L)}\left(Q(\log B)+O\left(B^{-\theta}\right)\right)$ for some polynomial $Q$ of degree at most $b(\mathbb{Q},(X,M),L)-1$ for which the   holds. In the second part we use \cite[Théorème 1]{Bre01Sum} to show that $Q$ has degree exactly equal to $b(\mathbb{Q},(X,M),L)-1$, and we compute the leading coefficient under assumption that $L$ is toric adjoint rigid.
	\subsection{Acknowledgments}
	This article is based on my PhD thesis \cite{Thesis}, which I wrote at
	Utrecht University under the supervision of Marta Pieropan. I would like to thank her for her
	unwavering support for my work and her various suggestions. I would also like to thank Alec
	Shute, Sam Streeter, Tim Santens and Damaris Schindler for the useful discussions we had
	during the workshop \textit{Campana points on toric varieties II} in Bristol, as well as on other occasions. I would also like to thank Régis de la Bretèche for his thoughtful response to a question of mine. Finally, I would like to thank Sho Tanimoto and Sam Streeter for their useful comments on an earlier version of the paper.
	
	\subsection{Notation}
	We use the convention that the set of natural numbers $\mathbb{N}$ contains $0$ and we write $\mathbb{N}^*$ for the set of nonzero natural numbers.

	We typically denote vectors using boldface and write their components using a normal face together with an index. For example, we may write $\mathbf{s}=(s_1,\dots,s_n)$ for a vector in $\mathbb{R}^n$. For two vectors $\mathbf{a}, \mathbf{b}\in \mathbb{R}^n$, we write $\mathbf{a}>\mathbf{b}$ if $a_i>b_i$ for all $i=1,\dots,n$. We also denote the $i$-th basis vector of $\mathbb{R}^n$ by $\mathbf{e}_i$.
	
	For an abelian group $G$, we write $G_{\mathbb{Q}}=G\otimes_{\mathbb{Z}} \mathbb{Q}$ and $G_{\mathbb{R}}=G\otimes_{\mathbb{Z}} \mathbb{R}$ for its base change to $\mathbb{Q}$ and $\mathbb{R}$, respectively. We denote its dual by $G^\vee=\Hom(G,\mathbb{Z})$. For a symbol $D$ we write $\mathbb{Z}(D)$ for the group isomorphic to $\mathbb{Z}$ with generator $D$, and similarly we write $\mathbb{Q}(D)\cong \mathbb{Q}$ for the vector space with generator $D$. 
	
	The logarithm $\log$ refers to the natural logarithm.
	For two nonnegative functions $f,g$ the notation $f(x)\ll g(x)$ means that there exists a constant $c$ such that $f(x)\leq cg(x)$ for all $x$ in the common domain of the functions $f$ and $g$.
	

For a field $K$ and a $\mathbb{Z}$-scheme $X$, we denote its base change to $K$ by $X_K=X\times_{\mathbb{Z}}\spec K$.
	\section{Geometry of toric pairs} \label{section: geometry toric pairs}
	In this section, we study the geometry of toric pairs and their Picard groups. We will first recall some basic toric geometry.
	
	\subsection{Cox coordinates on split toric varieties} \label{section: Cox coordinates}
	In this section we introduce smooth proper toric varieties and their Cox coordinates, which are a generalization of the homogeneous coordinates on projective space. We will use these Cox coordinates to describe and study $\cM$-points on toric pairs. Let $N$ be a lattice and let $\Sigma$ be a regular and complete fan in $N$ as in \cite[Definition 8.1.1]{Sal98}. To such a fan, we define the associated toric scheme (over $\mathbb{Z}$) to be $\cX_{\Sigma}$ as in \cite[Remark 8.6]{Sal98} and the associated toric variety over a field $K$ to be $X_{\Sigma}:=\cX_{\Sigma,K}$. By construction, $\cX_{\Sigma}$ is a proper and smooth $\mathbb{Z}$-scheme and $X:=X_{\Sigma}$ is a proper and smooth split toric variety over $K$.
	
	\begin{notation}
		For a fan $\Sigma$, we write $\Sigma(1)$ for the set of rays in $\Sigma$ and $\Sigma_{\max}$ for the set of maximal cones in $\Sigma$. For a torus invariant divisor $D_i\subset X$, we write $\rho_i\subset N_{\mathbb{R}}$ for the corresponding ray and $n_{\rho_i}\in N$ for its ray generator. Finally, we write $\Div_T(X)\cong \mathbb{Z}^n$ for the group of torus invariant divisors on $X$.
	\end{notation}
	
	\begin{definition}
		For a split toric variety $X=X_{\Sigma}$ over $\mathbb{Q}$, we define its \textit{toric integral model} to be $\cX=\cX_{\Sigma}$. We write $U=X\setminus \bigcup_{i=1}^n D_i$ for the dense open torus in $X$.
	\end{definition}
	On a split toric variety, rational points can be described using Cox coordinates, generalizing the homogeneous coordinates on projective space. We recall the construction of Cox coordinates on such varieties and their integral model. For a regular and complete fan $\Sigma$ in $N\cong \mathbb{Z}^d$, \cite[Theorem 4.1.3]{CLS11} gives an exact sequence
	\begin{equation} \label{equation: Picard group toric variety}
		0\rightarrow N^\vee\rightarrow \Div_T(X)\rightarrow \Pic(X)\rightarrow 0,
	\end{equation}
	where the first map is the map sending a character to its corresponding divisor. When combined with \cite[Proposition 4.2.5]{CLS11}, this sequence implies $\Pic(X)\cong \mathbb{Z}^{n-d}$. Furthermore, by this exact sequence, we can view the $\mathbb{Z}$-torus $\cG:=\Hom(\Pic(X), \mathbb{G}_{m,\mathbb{Z}})\cong \mathbb{G}_{m,\mathbb{Z}}^{n-d}$ as the kernel of the homomorphism $$\mathbb{G}_{m,\mathbb{Z}}^n\cong \Hom(\mathbb{Z}^{\Sigma(1)}, \mathbb{G}_{m,\mathbb{Z}}) \rightarrow \mathbb{G}_{m,\mathbb{Z}}^d\cong \cHom(N^\vee, \mathbb{G}_{m,\mathbb{Z}})$$
	induced by $N^\vee\rightarrow \mathbb{Z}^{\Sigma(1)}$.
	Now for each cone $\sigma\in \Sigma$, let
	$$\bfx^{\hat{\sigma}}:=\prod_{\substack{i=1 \\ \rho_i\not\in\sigma}}^n x_i,$$
	and let $$\cZ=\{(x_1,\dots,x_n)\in \mathbb{A}_{\mathbb{Z}}\mid\bfx^{\hat{\sigma}}=0 \text{ for all } \sigma\in \Sigma_{\max}\}.$$
	The torus $\cG$ acts on $\mathbb{A}^n_{\mathbb{Z}}\setminus \cZ$ by coordinate multiplication. The homomorphism $\phi\colon \mathbb{Z}^n\rightarrow N$ given by $\bfe_i\rightarrow n_{\rho_i}$ induces a morphism $$\pi\colon \mathbb{A}^n_{\mathbb{Z}}\setminus \cZ\rightarrow \cX_{\Sigma}$$ of toric varieties. As in \cite[Theorem 5.1.11]{CLS11}, this homomorphism is a $\cG$-torsor, so $\pi$ gives an isomorphism
	$$\cX_{\Sigma}\cong  \left(\mathbb{A}^n_{\mathbb{Z}}\setminus \cZ\right)/\cG.$$
	In fact, $\pi$ is a universal torsor for $\cX_{\Sigma}$.
	We use this isomorphism to define Cox coordinates.
	By \cite[Proposition 2.1]{FrPi13}, any point $P\in \cX_{\Sigma}(\mathbb{Z})$ is the image of a point $\tilde{P}\in \left(\mathbb{A}^n_{\mathbb{Z}}\setminus \cZ\right)(\mathbb{Z}).$
	\begin{definition}
		For a point $\tilde{P}=(a_1,\dots,a_n)\in \left(\mathbb{A}^n_{\mathbb{Z}}\setminus \cZ\right)(\mathbb{Z})$, we say that $(a_1,\dots,a_n)$ are (integral) \textit{Cox coordinates} for the point $P=\pi(\tilde{P})\in \cX_{\Sigma}(\mathbb{Z})=X_{\Sigma}(\mathbb{Q})$. We will write $P=(a_1:\dots:a_n)$, in analogy to the homogeneous coordinates on projective space. For $p$-adic points, we define Cox coordinates analogously.
	\end{definition}
	In particular, a point on a toric variety $P\in X(\mathbb{Q})$ lies on the torus $U(\mathbb{Q})$ exactly if it is described using nonzero Cox coordinates.
	Since $\cG(\mathbb{Z})\cong\{\pm 1\}^{n-d}$, a point on $U(\mathbb{Q})$ has exactly $2^{n-d}$ representatives by (integral) Cox coordinates.
	
	\subsection{Toric pairs and \texorpdfstring{$\cM$}{M}-points}

\begin{definition} \label{def: toric pair}
	A (smooth split) \textit{toric pair} over a field $K$ is a divisorial pair $(X,M)$ as in \cite{Moe25conjecture}, where $X$ is a smooth proper split toric variety over $K$ and $M=((D_1,\dots,D_n), \fM)$, where $D_1,\dots,D_n$ are the torus invariant prime divisors on $X$ and $\fM\subset \mathbb{N}^n$ is a subset containing the origin. The \textit{toric integral model} of a toric pair $(X,M)$ is the the pair $(\cX,\cM)$ over $\mathbb{Z}$, where $\cX$ is the toric integral model of $X$ and $\cM=((\cD_1,\dots,\cD_n), \fM)$, where $\cD_1,\dots,\cD_n$ are the Zariski closures of $D_1,\dots,D_n$ in $\cX$.
\end{definition}
In contrast to \cite{Moe25conjecture}, we restrict $\fM$ to be a subset of $\mathbb{N}^n$ as we will only count rational points contained in the dense torus.
Proper toric pairs are rationally connected by \cite[Proposition 5.3]{Moe25conjecture}, so they satisfy the hypotheses of \cite[Conjecture 1.4]{Moe25conjecture}.
\begin{notation}
	For a toric pair $(X,M)$, we let $\fM_{\red}\subset \fM$ be the collection of all $\bfm$ such that $\bigcap_{\substack{i=1 \\ m_i>0}}D_i\neq \emptyset$. Furthermore, we let $\Gamma_M$ be the minimal set of generators of $\fM_{\red}$. 
\end{notation}
For instance, if $X=\mathbb{P}^{n-1}$ and $\fM=\mathbb{N}^n$, then $\fM_{\red}=\{\bfm\in \mathbb{N}^n\mid \min(\bfm)=0\}$ and $\Gamma_M=\{\mathbf{e}_1,\dots,\mathbf{e}_n\}$.

	Using Cox coordinates, we define the multiplicity of every rational point with respect to a prime.
	\begin{definition}
		Let $X$ be a smooth proper toric variety over $\mathbb{Q}$ and let $p$ be a prime number. For every rational point $Q=(a_1:\dots:a_n)\in U(\mathbb{Q}_p)\subset \cX(\mathbb{Z}_p)$, we let the \textit{multiplicity} of $Q$ at $p$ be $$\mult_p(Q)=(v_p(a_1),\dots,v_p(a_n))\in \mathbb{N}^n.$$
	\end{definition}
	This multiplicity is the same as the multiplicity defined in \cite[Definition 3.8]{Moe25conjecture} determined by the divisors $\cD_1,\dots,\cD_n$.
	\begin{definition}
		Let $(X,M)$ be a toric pair over $\mathbb{Q}$ with toric integral model $(\cX,\cM)$, let $S$ be a positive integer and let $p$ be a prime number. A $p$-adic point $Q\in X(\mathbb{Q}_p)$ is a $p$-adic $\cM$-point if $\mult_p(Q)\in \fM$. A rational point $Q\in X(\mathbb{Q})$ is an $\cM$-point with respect to $\mathbb{Z}[\tfrac{1}{S}]$ if it is a $p$-adic $\cM$-point for all prime numbers $p$ not dividing $S$. We denote the sets of $\cM$-points by $(\cX,\cM)(\mathbb{Z}_p)$ and $(\cX,\cM)(\mathbb{Z}[\tfrac{1}{S}])$, respectively.
	\end{definition}
	\subsection{Images of \texorpdfstring{$\cM$}{M}-points under toric rational maps}
	We now describe how images of rational points under toric rational maps can be viewed as $\cM$-points on a toric pair.
	For this we first recall the homomorphism $\phi$ from \cite[Definition 5.10]{Moe24}.
\begin{definition} \label{def: invariants}
	Let $X$ be a smooth split toric variety with torus invariant prime divisors $D_1,\dots,D_n$ and corresponding ray generators $n_{\rho_1},\dots, n_{\rho_n}$. We define $\phi\colon \mathbb{Z}^n\rightarrow N$ to be the group homomorphism given by $$(m_1,\dots,m_n)\rightarrow \sum_{i=1}^n m_i n_{\rho_i}.$$
	For a toric pair $(X,M)$, we let $N_M^+,N_M$ be the monoid and the lattice generated by the image of $\fM_{\red}$, respectively.
\end{definition}
 Let $f\colon X\dashrightarrow X'$ be a toric rational map of smooth proper split toric varieties over $\mathbb{Q}$, i.e. a homomorphism $U\rightarrow U'$ of the dense tori in $X$ and $X'$, and let $\overline{f}\colon N\rightarrow N'$ be the corresponding map of lattices.
 \begin{definition} \label{def: image pair}
 	For a toric pair $(X,M)$, let $(X',f_*M)$ be the toric pair given by the set of multiplicities
 	$$f_* \fM=(\phi')^{-1}(\overline{f}\circ\phi)(\fM_{\red}),$$
 	where $\phi\colon \mathbb{Z}^n\rightarrow N, \phi'\colon \mathbb{Z}^{n'}\rightarrow N'$
 	are the homomorphisms from Definition \ref{def: invariants} for $X$ and $X'$.
 \end{definition}
The images of $\cM$-points under $f$ are $f_*\cM$-points, and the converse holds up to changing the signs of the coordinates.
\begin{lemma} \label{lemma: image pair}
	Let $f\colon X\dashrightarrow X'$ be a toric rational map of smooth split toric varieties over $\mathbb{Q}$ and let $(X,M)$ be a toric pair. Let $\cU,\cU'$ be the dense tori of $\cX$ and $\cX'$. Then
	$$f((\cX,\cM)(\mathbb{Z}))\subset (\cX',f_*\cM)(\mathbb{Z}).$$
	Furthermore, for every $P\in (\cX',f_*\cM)(\mathbb{Z})$, there exists $u\in \cU'(\mathbb{Z})$ such that the product $uP$ lies in the image, and $u$ is unique modulo $f(\cU(\mathbb{Z}))$.
\end{lemma}
\begin{proof} Let $p$ be a prime number. The composition $\phi\circ\mult_p\colon U(\mathbb{Q}_p)\rightarrow N$ is a surjective group homomorphism with kernel $\cU(\mathbb{Z}_p)$, see e.g. \cite[Proposition 5.12.]{Moe24}, giving a natural isomorphism $N\cong U(\mathbb{Q}_p)/\cU(\mathbb{Z}_p)$.
Therefore, the diagram
$$\begin{tikzcd}
	U(\mathbb{Q}_p) \arrow[r, "\mult_p"] \arrow[d] & \mathbb{N}^n \arrow[r, "\phi"]     & N \arrow[d] \\
	U'(\mathbb{Q}_p) \arrow[r, "\mult_p"]          & \mathbb{N}^{n'} \arrow[r, "\phi'"] & N'         
\end{tikzcd}$$
commutes. Therefore for every $Q\in (\cX,\cM)(\mathbb{Z}_p)$, we have that $\mult_p(f(Q))\in f_*\fM$. Furthermore, the Chinese remainder theorem implies that for every $P\in (\cX',f_*\cM)(\mathbb{Z})$ there is $Q\in (\cX,\cM)(\mathbb{Z})$ such that $f(Q)$ has the same multiplicities as $P$ at all prime numbers. Therefore, they differ by an $\mathbb{Z}$-integral point $u$ on the torus $\cU'$. Finally, if is another integral point $u'$ such that $u'P$ lies in the image, then $u^{-1}u'$ also lies in the image.
\end{proof}
\begin{remark}
Lemma \ref{lemma: image pair} generalizes from $\mathbb{Z}$ to any ring of $\cO_{K,S}$ integers, over a global field $K$, provided $S$ is chosen large enough to ensure that $\cO_{K,S}$ is a principal ideal domain. This is done by replacing the role of the Chinese remainder theorem in the proof with the fact that $\mathbb{A}^1$ satisfies strong approximation away from the infinite places.
\end{remark}
\begin{example}
Let $X$ be a smooth proper split toric variety over $\mathbb{Q}$, let $m$ be a positive integer, and let $f\colon X\rightarrow X$ be the morphism corresponding to the multiplication-by-$m$ map $N\rightarrow N$. Then $f(X(\mathbb{Q}))$ are the set of points in $X(\mathbb{Q})$ of the form $(a_1^{m}:\dots:a_n^m)$, for nonzero rational numbers $a_1,\dots,a_n$. On the other hand, if $(X,M)$ is the pair given by $\fM=\mathbb{N}^n$, then $(\cX,f_*\cM)(\mathbb{Z})$ are the strict Darmon points corresponding to the Campana pair $\left(X,\sum_{i=1}^n \left(1-\tfrac{1}{m}\right)D_i\right)$ (as defined in \cite[Definition 3.16]{Moe24}). These are the points of the form $(\pm a_1^{m}:\dots:\pm a_n^m)$, for nonzero rational numbers $a_1,\dots,a_n$. In particular we see that the set of Darmon points is equal to the image exactly when $m$ is odd.
\end{example}
	\subsection{Divisors on toric pairs}
	In this section, we will study the Picard group of toric pairs. 
	\begin{definition}
		We write $\Div_T(X)$ for the group of torus-invariant divisors on a split toric variety $X$. Analogously, we call an element of $\Div_T(X,M):=\bigoplus_{\bfm\in \Gamma_M} \mathbb{Z}(\tilde{D}_\bfm)$ a \textit{torus-invariant} divisor on $(X,M)$. Similarly, we call an element in $\Div_T(X,M)_{\mathbb{Q}}$ a torus-invariant $\mathbb{Q}$-divisor.
	\end{definition}
	
	There is a natural homomorphism from torus-invariant divisors on $X$ to torus-invariant divisors on $(X,M)$.
	\begin{definition}
		Let $(X,M)$ be a toric pair over a field $K$.
		The \textit{pullback} of torus-invariant divisors on $X$ to torus-invariant divisors on $(X,M)$ is the group homomorphism $$\pr_M^*\colon \Div_T(X)\rightarrow \Div_T(X,M)$$ defined by
		$$D_i\mapsto \sum_{\bfm\in \Gamma_{M}} m_i\tilde{D}_{\bfm},$$
		for $i=1,\dots, n$.
	\end{definition}
	This is the restriction of the pullback homomorphism in \cite[Definition 4.15]{Moe25conjecture} to torus-invariant divisors.
	Using the pullback, we define the Picard group of a toric pair.
	
	\begin{definition}
		Let $(X,M)$ be a toric pair. A torus-invariant divisor on $(X,M)$ is \textit{principal} if it is the pullback of a principal divisor on $X$. We say that two divisors $D,D'$ on $(X,M)$ are \textit{linearly equivalent} if $D-D'$ is principal.
		We define the \textit{Picard group} of $(X,M)$ as
		$$\Pic(X,M)=\Div_T(X,M)/\{\text{principal divisors}\}.$$
	\end{definition}
	By definition, the pullback $\pr^*_M \colon \Div_T(X)\rightarrow \Div_T(X,M)$ induces a homomorphism
	$$\Pic(X)\rightarrow \Pic(X,M).$$
	Finally, we introduce the effective cone.
	\begin{definition}
		A torus-invariant divisor $\sum_{\bfm\in \Gamma_M} a_{\bfm}D_{\bfm}$ on a toric pair $(X,M)$ is \textit{effective} if $a_{\bfm}\geq 0$ for all $\bfm\in \Gamma_M$. We call the cone $\Eff^1(X,M)\subset \Pic(X,M)_{\mathbb{R}}$ generated by torus-invariant divisors the \textit{effective cone}.
	\end{definition}
	
	In \cite{Moe25conjecture}, we defined the Picard group and the effective cone more generally by using the divisor group $$\Div(X,M)=\Div(U)\oplus \Div_T(X,M)$$ of all divisors on $(X,M)$, rather than using only the torus-invariant divisors. However, the next proposition shows that every effective divisor on $(X,M)$ is linearly equivalent to an effective torus-invariant divisor, so the definitions agree.
	\begin{proposition} \label{prop: effective cone generated by torus-invariant}
		Let $(X,M)$ be a smooth toric pair.
		Every divisor $D$ on $(X,M)$ is linearly equivalent to a torus-invariant divisor on $(X,M)$. Furthermore, if $D$ is effective, it is linearly equivalent to an effective torus-invariant divisor. Hence the effective cone $\Eff^1(X,M)$ is a rational polyhedral cone generated by the torus-invariant prime divisors on $(X,M)$.
	\end{proposition}
	\begin{proof}
		By \cite[Theorem 4.1.3]{CLS11} every divisor on $X$ is linearly equivalent to a torus-invariant divisor. This directly implies that every divisor on $(X,M)$ is linearly equivalent to a torus-invariant divisor on $(X,M)$.
		
		The proof of the statement for effective divisors is based on the proof of the analogous statements for divisors on $X$ given in \cite[Proposition 4.3.2, Lemma 15.1.8]{CLS11}.
		For an effective divisor $D$ on $(X,M)$, let $D'$ be a torus-invariant divisor linearly equivalent to it. Let
		$$H^0(X,D')=\{f\in K(X)^\times \mid D'+\pr^*_M\div f\geq 0\}\cup\{0\},$$
		where we write $E\geq 0$ for a divisor $E$ on $(X,M)$ to indicate that the divisor is effective.
		Then $H^0(X,D')$ is a vector space invariant under the natural torus action on $K(X)^\times$. Hence by \cite[Lemma 1.1.16]{CLS11},
		$$H^0(X,D')=\bigoplus_{\substack{\mu \in N^\vee\\ D'+\pr^*_M\div \chi^\mu\geq 0 }} K\cdot \chi^\mu,$$
		where $\chi^{\mu}\in \cO(U)^\times$ is the character determined by $\mu \colon N\rightarrow \mathbb{Z}$.
		Since $D$ is effective, $H^0(X,D')$ has to be nontrivial, and thus there exists $\mu\in N^\vee$ such that $D'+\pr^*_M \div(\chi^\mu)$ is an effective torus-invariant divisor linearly equivalent to $D$.
	\end{proof}
	Using the map $\phi\colon \mathbb{Z}^n\rightarrow N$ introduced in Definition \ref{def: invariants} and an analogue of the exact sequence \eqref{equation: Picard group toric variety}, we will give a simple presentation for the Picard group of a toric pair.
	\begin{proposition} \label{prop: Picard group toric pair}
		Let $(X,M)$ be a smooth toric pair over a field $K$. There is an exact sequence
		$$N^\vee\rightarrow \Div_T(X,M)\rightarrow \Pic(X,M)\rightarrow 0,$$
		where $N^\vee\rightarrow \Div_T(X,M)$ is the composition of the map $N^\vee\rightarrow \Div_T(X)$ with the pullback map $\Div_T(X)\rightarrow \Div_T(X,M)$, where the first map is given by sending a character to its corresponding divisor.
		Furthermore, we have an exact sequence
		$$0\rightarrow N^\vee\rightarrow \Div_T(X,M)\rightarrow \Pic(X,M)\rightarrow 0$$
		if and only if the lattice $N_M$ from Definition \ref{def: invariants} has finite index in $N$. 
	\end{proposition}
	\begin{proof}
		
		The torus-invariant principal divisors on $(X,M)$ are exactly the pullbacks of torus-invariant principal divisors on $X$. Thus, the kernel of this homomorphism is the image of $N^\vee$ in $\Div_T(X,M)$ by \cite[Theorem 4.1.3]{CLS11}.
		
		The map $N^\vee\rightarrow \Div_T(X,M)$ is given by
		$$\mu\mapsto \pr^*_M\div(\chi^\mu)= \sum_{\bfm\in \Gamma_M} \langle \mu, \phi(\bfm) \rangle\tilde{D}_{\bfm},$$
		where $\phi$ is the homomorphism defined in Definition \ref{def: invariants}.
		This directly implies that $N^\vee\rightarrow \Div_T(X,M)$ is injective if and only if $\{\phi(\bfm)\mid \bfm\in \Gamma_M\}$ spans $N_\mathbb{Q}$ as a vector space, which is equivalent to $N_M$ having finite index in $N$.
	\end{proof}
	In \cite[Definition 4.33]{Moe25conjecture}, we introduced the \textit{canonical divisor class} $K_{(X,M)}\in \Pic(X,M)$ of a pair. On toric pairs, this class has a particularly nice representative.
	\begin{proposition} \label{prop: canonical divisor toric pair}
		Let $(X,M)$ be a smooth toric pair over a field of characteristic $0$. The canonical class of $(X,M)$ as is given as
		$$K_{(X,M)}=-\sum_{\bfm\in \Gamma_M}[\tilde{D}_\bfm].$$
	\end{proposition}
	\begin{proof}
		By \cite[Theorem 8.2.3]{CLS11}, the canonical divisor class of a toric variety is
		$$K_X=-\sum_{i=1}^n [D_i],$$
		which directly implies the desired identity.
	\end{proof}
	\begin{notation} \label{notation: toric canonical divisor}
		We will denote the representative of $K_{(X,M)}$ introduced in Proposition \ref{prop: canonical divisor toric pair} by $$D_{(X,M)}=-\sum_{\bfm\in \Gamma_M}\tilde{D}_\bfm\in \Div_T(X,M).$$
	\end{notation}
	We recall the Fujita invariant and the $b$-invariant from \cite{Moe25conjecture}.
	\begin{definition} \label{def: a and b invariant}
		Let $(X,M)$ be a toric pair and let $L$ be a big and nef $\mathbb{Q}$-divisor class on $X$. We define the \textit{Fujita invariant} of $(X,M)$ with respect to $L$ to be
		$$a((X,M),L)=\inf\{t\in \mathbb{R}\mid t\pr^*_M L+K_{(X,M)}\in \Eff^1(X,M)\}.$$
		We call $a((X,M),L)\pr^*_M L+K_{(X,M)}$ the \textit{adjoint divisor class of $L$ with respect to $(X,M)$}. We define the $b$-invariant $b(\mathbb{Q},(X,M),L)$ to be the codimension of the minimal supported face of $\Eff^1(X,M)$ containing the adjoint divisor class of $L$ with respect to $(X,M)$.
	\end{definition}
	
	Note that there need not exist a $\mathbb{Q}$-divisor class $L$ such that $\pr^*L=-K_{(X,M)}$. In particular, the $b$-invariant can be strictly smaller than the Picard rank of $(X,M)$ for all choices of $L$. For instance this happens for pairs corresponding to Campana points, see \cite[Example 4.38]{Moe25conjecture}
	
	In \cite{Moe25conjecture}, rigid divisors on pairs were defined.
	An effective divisor on a pair $(X,M)$ is \textit{rigid} if it is the only effective divisor in its linear equivalence class \cite[Definition 4.44]{Moe25conjecture}. In the proof of Theorem \ref{theorem: full asymptotics}, we will consider rigid divisors as well as divisors which are rigid with respect to torus-invariant divisors.
	\begin{definition} \label{def: toric adjoint rigid}
		Let $D\in \Div_T(X,M)_\mathbb{Q}$ be an effective $\mathbb{Q}$-divisor on a smooth toric pair $(X,M)$. We say that $D$ is \textit{(toric) rigid} if $D$ is the only effective (torus-invariant) $\mathbb{Q}$-divisor in its $\mathbb{Q}$-linear equivalence class. For a big and nef $\mathbb{Q}$-divisor $L$ on $X$, we say that $L$ is \textit{(toric) adjoint rigid} with respect to $(X,M)$ if the adjoint divisor class $a((X,M),L)\pr_M^* L+K_{(X,M)}$ is represented by a (toric) rigid effective $\mathbb{Q}$-divisor.
	\end{definition}
	For many toric pairs, toric rigid divisors are just the same as rigid divisors.
	\begin{proposition}
		Let $(X,M)$ be a smooth toric pair. Then every rigid $\mathbb{Q}$-divisor is toric rigid.
		If the pullback map $\Div_T(X)\rightarrow \Div_T(X,M)$ is injective, then the converse also holds.
	\end{proposition}
	The pullback map is injective on a proper toric pair by \cite[Proposition 4.30.]{Moe25conjecture}. Thus, on such a pair, toric rigid divisors are the same as rigid divisors.
	\begin{proof}
		If an effective $\mathbb{Q}$-divisor $D\in \Div(X,M)_{\mathbb{Q}}$ is rigid, then Proposition \ref{prop: effective cone generated by torus-invariant} implies that it has to be torus-invariant, and thus toric rigid.
		
		For the converse, we argue by proof by contrapositive. Let $D\in \Div_T(X,M)$ be an effective divisor which is not rigid. Then the vector space
		$$H^0(X,kD)=\{f\in K(X)^\times \mid D+\pr^*_M\div f\geq 0\}\cup\{0\}$$
		considered in Proposition \ref{prop: effective cone generated by torus-invariant}
		is at least two dimensional for some positive integer $k$.
		As
		$$H^0(X,kD)=\bigoplus_{\substack{\mu \in N^\vee\\ kD+\pr^*_M\div \chi^\mu\geq 0 }} K\cdot \chi^\mu,$$
		this implies that there exist at least one nonzero $\mu\in N^\vee$ such that $kD+\pr^*_M\div \chi^\mu\geq 0$. If $\Div_T(X)\rightarrow \Div_T(X,M)$ is injective, then $\div \chi^\mu\neq 0$ so $D+\div \chi^\mu$ is an effective torus-invariant divisor linearly equivalent to $D$ and thus $D$ is not toric rigid.
	\end{proof}
	In general there may be more toric rigid divisors than rigid divisors, as the next examples show.
	\begin{example}
		Let $(X,M)$ be the toric pair given by $X=\mathbb{P}^1$ and $\fM=\{(0,0)\}$, i.e., the pair corresponding to integral points on the torus $\mathbb{G}_m\subset \mathbb{P}^1$. Then $\Div_T(X,M)=0$ so the trivial divisor on $(X,M)$ is toric rigid. On the other hand, $\Pic(\mathbb{G}_m)=0$ so any effective divisor on $(X,M)$ is linearly equivalent to the trivial divisor, and thus the trivial divisor is not rigid.
	\end{example}
	\begin{example}
		Let $X=\mathbb{P}^1\times \mathbb{P}^1$ and let $(X,M)$ be the toric pair corresponding to the integral points for the open $\mathbb{G}_m\times \mathbb{P}^1\subset X$. Then the previous example implies that any big and nef $\mathbb{Q}$-divisor on $X$ is toric adjoint rigid with respect to $(X,M)$.
	\end{example}
	An analogous statement is true for Hirzebruch surfaces of higher degree.
	\begin{example}
		Let $X=\Proj_{\mathbb{P}^1} (\cO_{\mathbb{P}^1}\oplus \cO_{\mathbb{P}^1}(d))$ be the Hirzebruch surface of degree $d>0$. Let $D_2$ be the unique prime divisor with self-intersection $D_2^2=-d$, and let $D_1,D_3$ be the torus-invariant prime divisors intersecting the prime divisor $D_2$. Let $(X,M)$ be the toric pair corresponding to the open $U'=X\setminus (D_1\cup D_3)$. We identify $\Pic(X,M)$ with $\Pic(U')$ using the natural isomorphism as in \cite[Example 4.26]{Moe25conjecture}. Under this identification, the anticanonical divisor on $(X,M)$ is $-K_{(X,M)}=[D_2]+[D_4]\in \Pic(U')$, which is a nonzero effective divisor class. Since $\Pic(U')=\mathbb{Z}$, every divisor on $U'$ is linearly equivalent to a rational multiple of $-K_{(X,M)}$. Consequently, any big and nef $\mathbb{Q}$-divisor on $X$ is toric adjoint rigid with respect to $(X,M)$.
	\end{example}
	\subsection{Quasi-proper pairs}
	In order to prove Theorem \ref{theorem: asymptotics proper toric} for a divisor class $L$ which is not adjoint rigid, we need to count the points on a smaller \textit{quasi-proper} pair.
\begin{definition} \label{def: quasi-proper pair}
	Let $X$ be a smooth proper variety and let $L\in \Pic(X)_{\mathbb{Q}}$ be a big and nef $\mathbb{Q}$-divisor class. A smooth pair $(X,M)$ is \textit{quasi-proper with respect to $L$} if there exists a smooth proper pair $(X,\tilde{M})\supset (X,M)$ such that $a((X,M),L)=a((X,\tilde{M}),L)$. Here the containment is as in \cite[Definition 3.12]{Moe25conjecture} (which for toric pairs boils down to $\fM\subset \tilde{\fM}$).
\end{definition}
Such a pair $(X,\tilde{M})$ can be found by adjoining sufficiently large multiples of the basis vectors to $\fM$.
\begin{proposition} \label{prop: choice quasi-proper pair}
	Let $(X,M)$ be a smooth divisorial pair which is quasi-proper with respect to some big and nef $\mathbb{Q}$-divisor class. Then a proper pair $(X,\tilde{M})$ as in Definition \ref{def: quasi-proper pair} can be found by taking $\tilde{M}=((D_1,\dots,D_n), \tilde{\fM})$, where $$\tilde{\fM}=\fM\cup \{m\bfe_1,\dots,m\bfe_n\}$$ for any sufficiently large integer $m$.
\end{proposition}
\begin{proof}
	For any proper divisorial pair $(X,\overline{M})$ satisfying $a((X,M),L)=a((X,\overline{M}),L)$, there are positive integers $m_1,\dots,m_n$ such that $m_1\bfe_1,\dots,m_n\bfe_n\in \overline{\fM}$, and \cite[Proposition 4.41]{Moe25conjecture} implies that the pair $(X,M')$ given by $\fM'=m_1\bfe_1,\dots,m_n \bfe_n\}\cup \fM$ has the same Fujita invariant as $(X,\overline{M})$ and $(X,M)$. Now let  $m\geq m_1,\dots,m_n$ be an integer and let $(X,\tilde{M})$ be the pair given by $\tilde{\fM}_{\cC}=\{m\bfe_1,\dots,m\bfe_n\}\cup \fM$. The pair $(X,M'')$ given by $\fM''= \fM'\cup \tilde{\fM}$ has the same Fujita invariant as $(X,M')$ by \cite[Lemma 4.39]{Moe25conjecture}, and \cite[Proposition 4.41]{Moe25conjecture} implies
	$$a((X,\tilde{M}),L) \leq a((X,M''),L)=a((X,M),L).$$
	Now since $(X,M)\subset (X,\tilde{M})$, we must have $a((X,\tilde{M}),L)=a((X,M),L)$ by \cite[Proposition 4.41]{Moe25conjecture}.
\end{proof}
For a pair $(X,\tilde{M})$ as in Proposition \ref{prop: choice quasi-proper pair}, there is an inclusion $\Gamma_{M}\subset \Gamma_{\tilde{M}}$ giving a restriction homomorphism  $\Pic(X,\tilde{M})\rightarrow \Pic(X,M)$ as in \cite[Definition 4.27]{Moe25conjecture}. We can view the condition $a((X,M),L)=a((X,\tilde{M}),L)$ as the statement that the adjoint divisor of $L$ with respect to $(X,M)$ is the restriction of the adjoint divisor of $L$ with respect to $(X,\tilde{M})$.

\begin{example}
	Let $X$ be a rationally connected smooth proper variety such that $-K_X$ is big, and let $D$ be a strict normal crossings divisor on $X$ which is rigid. Then the pair $(X,M)$ corresponding to the integral points on the open $U=X\setminus D$ is quasi-proper for any any big and nef $\mathbb{Q}$-divisor class $L=-K_X+aD$ with $a>-1$. We can take the proper pair $(X,\tilde{M})$ to be the pair corresponding to the Darmon points on $\left(X, \sum_{i=1}^n \left(1-\frac{1}{m}\right)D_i\right)$, where $D_1,\dots,D_n$ are the irreducible components of $D$ and $m$ is a positive integer such that $-1+\frac{1}{m}\leq a$. In particular, if $a\geq 0$, then we can take $(X,\tilde{M})$ to be the trivial pair. This follows from the rigidity of $D$ combined with the simple calculation
	$$\pr^*_{\tilde{M}}L+K_{(X,\tilde{M})}=\sum_{i=1}^n (am+(m-1))\tilde{D}_{m\bfe_i},$$
	which corresponds to $\sum_{i=1}^n \left(a+1-\frac{1}{m}\right)D_{i}$ under the identification in \cite[Example 4.25]{Moe25conjecture}. This implies $a((X,M),L)=a((X,\tilde{M}),L)=1$, as desired.
	
\end{example}
\begin{example}
	As a special case of the previous example, we can take $X=\Bl_{(0:0:1)}\mathbb{P}^2$, let $D_1$ be the exceptional divisor and take $\fM=\{0\}$. If $a\in (-1,1]\cap \mathbb{Q}$, then the $\mathbb{Q}$-divisor class $L=K_X+aD_1$ is big and nef, while $(X,M)$ is quasi-proper with respect to $L$.
\end{example}	
	\subsection{Fans and universal torsors for pairs} \label{section: universal torsor}
	In \cite[Section 11]{Sal98}, Salberger verifies Manin's conjecture for split toric varieties with the anticanonical height. In order to achieve this, he reduces the counting problem to estimating the volume of a domain $D(B)\subset Y(\mathbb{R})$ in the real locus of the universal torsor $Y$ of the variety.
	Afterwards, he uses the fan $\Sigma$ of $X$ to give a partition of $D(B)$ into simpler pieces: $D(B)=\bigcup_{\sigma\in \Sigma_{\max}}D(B,\sigma)$.
	In order to adapt the proof of Salberger to toric pairs, we thus need appropriate analogues of the fan and of the universal torsor for toric pairs.
	
	Using the map $\phi$ from Definition \ref{def: invariants} and the pair $(X,M)$, we subdivide the fan of $X$.
	\begin{notation} \label{notation: fan toric pair}
		Let $(X,M)$ be a toric pair over a field $K$. We let $(X,\overline{M})$ be a proper toric pair with $\overline{\fM}=\fM\cup\{d_i\mathbf{e}_i\mid m\mathbf{e}_i\not\in\fM \text{ for all } m\in \mathbb{N}^*\}$, where $\mathbf{e}_i\in \mathbb{N}^n$ denotes the $i$-th basis vector and $d_1,\dots,d_n$ are positive integers such that $\Gamma_M\subset \Gamma_{\overline{M}}$. Choose a simplicial fan $\Sigma_{\overline{M}}$ refining $\Sigma$ with rays
		$$\Sigma_{\overline{M}}(1)=\{\rho_\bfm\mid \bfm\in \Gamma_{\overline{M}}\},$$
		where $\rho_\bfm=\mathbb{R}_{\geq 0} \phi(\bfm)$. Here we recall that $\phi\colon \mathbb{N}^n\rightarrow N$ is the homomorphism given by $(m_1,\dots, m_n) \mapsto \sum_{i=1}^n m_i n_{\rho_i}$.
		Let $\Sigma_M\subset \Sigma_{\overline{M}}$ be the subfan given by the cones whose rays lie in $\{\rho_\bfm\mid \bfm\in \Gamma_M\}$.
		
		Similarly, for a maximal cone $\sigma\in \Sigma_{\overline{M}}$, let $(X,M(\sigma))$ be the toric pair given by $$\overline{\fM}=\fM\cup\{d_i\mathbf{e}_i\mid \rho_i\subset \sigma(1),\, m\mathbf{e}_i\not\in\fM \text{ for all } m\in \mathbb{N}^*\},$$ and we set $\Sigma_{M(\sigma)}\subset \Sigma_{\overline{M}}$ to be the subfan given by the cones whose rays lie in $\{\rho_\bfm\mid \bfm\in \Gamma_M\}\cup \sigma(1)$.
	\end{notation}
	Note that the fan $\Sigma_{\overline{M}}$ does not depend on the choice of the integers $d_1,\dots,d_n$ defining the pair $(X,\overline{M})$. 
	\begin{remark}
		If $(X,M)$ is a pair which is quasi-proper with respect to some $\mathbb{Q}$-divisor class $L$ and the integers $d_1,\dots,d_n$ are chosen sufficiently large, we have $a((X,M),L)=a((X,\overline{M}),L)$ by Proposition \ref{prop: choice quasi-proper pair}.
	\end{remark}
	\begin{remark}
		For a toric pair $(X,M)$ with a choice of a pair $(X,\overline{M})$ as above, there may be different non-isomorphic choices for the fan $\Sigma_{\overline{M}}$. Nevertheless, for many pairs, such as pairs corresponding to Campana points, there is only one such choice. For the purposes of this article, the choice of the fan is not important.
	\end{remark}
	Note that $\Sigma_{\overline{M}}$ is a complete fan, since $\Sigma$ is as well. Furthermore, if $(X,M)$ is proper, then $(X,\overline{M})=(X,M)$ so $\Sigma_{M}=\Sigma_{\overline{M}}$.
	
	We will also define universal torsors of toric pairs.
	Consider the morphism $f\colon \mathbb{A}_K^{\Gamma_M}\rightarrow \mathbb{A}_K^n$ given by sending $(x_\bfm)_{\bfm\in \Gamma_M}$ to $(y_1,\dots,y_n)$ where $y_i=\prod_{\bfm\in \Gamma_M}x_{\bfm}^{m_i}$ for all $i\in \{1,\dots,n\}$.
	\begin{definition} \label{def: universal torsor pair}
		We define the \textit{universal torsor $Y_M$ of $(X,M)$} as the open toric subvariety $Y_M=f^{-1}Y$ of $\mathbb{A}_K^{\Gamma_M}$, where $Y\rightarrow X$ is the universal torsor of $X$ as defined in Section \ref{section: Cox coordinates}. Let $U_M=\mathbb{G}_m^{\Gamma_M}$ be the dense torus of $Y_M$. The restriction of $Y_M\rightarrow X$ to $U_M\rightarrow U$ is a homomorphism of dense tori, and we write $T_M\subset Y_M$ for the kernel of this homomorphism.
	\end{definition}
	We call $Y_M$ the universal torsor of $(X,M)$, as it plays an analogous role to the universal torsor of $X$ in Salberger's work \cite{Sal98}.
	\begin{remark}
		By Proposition \ref{prop: Picard group toric pair}, the group variety $T_M$ is isomorphic to $\Hom(\Pic(X,M),\mathbb{G}_m)$. Therefore it can be seen as an analogue of the Picard torus of a toric variety, but it need not be a torus as $\Pic(X,M)$ can have torsion.
		If $(X,M)$ is a proper pair, then $f$ is surjective, so the natural morphism $Y_M\rightarrow X$ is surjective as well. Furthermore, Proposition \ref{prop: Picard group toric pair} implies that the restriction $U_M\rightarrow U$ is a $T_M$-torsor.
	\end{remark}
	For toric pairs which are not proper, the morphism $Y_M\rightarrow X$ need not be dominant and the analogy with the universal torsor of a toric variety partially breaks down, as the following example shows. We still maintain the same terminology however, to avoid unnecessary case distinctions.
	\begin{example}
		Let $X=\mathbb{P}^1$ and $\fM=\{(0,0)\}$, i.e., the pair corresponding to integral points on $\mathbb{G}_m$. Then $Y_M$ is just a point and the map $Y_M\rightarrow X$ is given by sending the point to $(1:1)$.
	\end{example}
	
	\section{\texorpdfstring{$\cM$}{M}-points of bounded height on toric varieties} \label{section: M points bounded height}
	\subsection{Heights} \label{section: heights on toric variety}
	Let $X$ be a smooth proper split toric variety over $\mathbb{Q}$ with fan $\Sigma$. Any $\mathbb{Q}$-divisor class $L$ on $X$ naturally induces a height function $H_{L}$ on the torus $U\subset X$, as defined for divisor classes by Batyrev and Tschinkel in \cite[Definition 2.1.7]{BaTs95}. In this section, we describe this height using its description given in \cite[Section 6.3]{PiSc24}.
	
	
	We can represent a $\mathbb{Q}$-divisor class $L$ by a $\mathbb{Q}$-divisor
	$$D=a_1D_1+\dots+a_nD_n,$$
	for some rational numbers $a_1,\dots,a_n\in \mathbb{Q}$. For a maximal cone $\sigma\in \Sigma$, write $$\mu_D(\sigma)=a_1\mu_{D_1}(\sigma)+\dots+a_n \mu_{D_n}(\sigma),$$ where $\mu_{D_i}(\sigma)\in N^\vee$ is the unique character of $U$ such that $\chi^{-\mu_{D_i}(\sigma)}$ generates $\mathcal{O}(D_i)$ on $U_\sigma$.
	Let $\sigma_1,\dots, \sigma_k$ be the maximal cones in the fan $\Sigma$. For a ray $\rho\in \Sigma$, we write $n_\rho$ for the corresponding ray generator.
	For $i=1,\dots,n$ and $j=1,\dots,k$ we define
	\begin{equation} \label{Equation: l^{(i)}}
		l^{(i)}(\mathbf{e}_j)=a_i-\langle \mu_D(\sigma_j),n_{\rho_i}\rangle,
	\end{equation}
	and 
	$$l^{(i)}(\bfs)=\sum_{j=1}^k l^{(i)}(\mathbf{e}_j)s_j$$
	for $\bfs\in \mathbb{C}^k$, where we recall that $\langle \cdot,\cdot\rangle\colon N^\vee_{\mathbb{Q}}\times N_{\mathbb{Q}}\rightarrow \mathbb{Q}$ is the natural pairing.
	
	Note that $L(\sigma_j):=\sum_{i=1}^n l^{(i)}(\mathbf{e}_j)D_i$ is $\mathbb{Q}$-linearly equivalent to $L$ by definition.
	All coefficients $l^{(i)}(\mathbf{e}_j)$ are nonnegative when $L$ is nef, as the following proposition shows.
	\begin{proposition}
		Let $L$ be a $\mathbb{Q}$-divisor on $X$. Then the following holds for any maximal cone $\sigma_j\in \Sigma$:
		\begin{enumerate}
			\item If $\rho_i\subset \sigma_j$, then $l^{(i)}(\mathbf{e}_j)=0$.
			\item If $L$ is nef, then $L$ is semiample and $l^{(i)}(\mathbf{e}_j)\geq 0$ for all $i=1,\dots, n$, so $L(\sigma_j)$ is effective.
		\end{enumerate}
	\end{proposition}
	\begin{proof}
		The first part follows directly from the definition of $l^{(i)}(\mathbf{e}_j)$. If $L$ is nef, then it is semiample by \cite[Theorem 6.3.12.]{CLS11}, and the nonnegativity of $l^{(i)}(\mathbf{e}_j)\geq 0$ is proved as in \cite[Proposition 8.7(a)]{Sal98}.
	\end{proof}
	We will now assume that $L$ is big and nef. As in \cite[Section 6]{PiSc24}, we define the function
	$$\bfx^{L(\sigma_j)}:=\prod_{i=1}^n x_i^{l^{(i)}(\mathbf{e}_j)},$$
	for every $j=1,\dots, k$ and $(x_1,\dots,x_n)\in Y(\mathbb{Q})$. Now \cite[Proposition 6.10]{PiSc24} implies that the height of a point $(x_1\colon \dots\colon x_n)\in X(\mathbb{Q})$ with respect to $L$, as defined in \cite[Section 6.3]{PiSc24}, is given as
	$$H_{L}(\bfx)=\prod_{v\in \Omega_{\mathbb{Q}}}\max(|\bfx^{L(\sigma_1)}|_{v},\dots, |\bfx^{L(\sigma_k)}|_{v}),$$
	where $\bfx=(x_1,\dots,x_n)\in Y(\mathbb{Q})$. Let $\cY\rightarrow \cX$ be the universal torsor of $\cX$ as introduced in Section \ref{section: Cox coordinates}.
	If we choose the coordinate representatives $(x_1,\dots,x_n)$ to be integers representing a $\mathbb{Z}$-integral point on $\cY$, then the formula for the height simplifies to 
	$$H_{L}(\bfx)=\max(|\bfx^{L(\sigma_1)}|,\dots, |\bfx^{L(\sigma_k)}|).$$
	
	
	\subsection{General asymptotics and the leading constant}
	In this section we present Theorem \ref{theorem: full asymptotics}, which will describe the asymptotic behavior of the toric counting function $N_{(X,M),L,S}(B)$ on a quasi-proper pair. To describe the leading constant appearing in the asymptotic, we first recall the $\alpha$-constant of the pair relative to $L$ as defined in \cite[Definiton 8.2]{Moe25conjecture}.
	\begin{definition} \label{def: alpha-constant}
		Let $(X,M)$ be a smooth toric pair, and let $L$ be a big and nef $\mathbb{Q}$-divisor on $X$ which is toric adjoint rigid with respect to $(X,M)$. Let $E$ be the unique effective $\mathbb{Q}$-divisor on $(X,M)$ with $\mathbb{Q}$-linear equivalence class $a((X,M),L)\pr_M^* L+K_{(X,M)}$, and let $(X,M^\circ)\subset (X,M)$ be the pair such that $\fM^\circ\setminus \{\mathbf{0}\}=\Gamma_{M^\circ}$ is the set of $\bfm\in \Gamma_M$ such that the associated divisor $\tilde{D}_{\bfm}$ is not contained in the support of $E$.
		
		Let $\Lambda=\Eff^1(X,M^\circ)$ be the effective cone, and let $\Lambda^\vee\subset \Pic(X,M^\circ)^\vee_{\mathbb{R}}$ be its dual cone. Then the \textit{$\alpha$-constant} of the pair $(X,M)$ with respect to $L$ is
		$$\alpha((X,M),L):=\frac{1}{\#\Pic(X,M^{\circ})_{\mathrm{torsion}}} \int_{\Lambda^\vee} e^{-\langle \pr_{M^\circ}^*(L),\bfx\rangle}\d \bfx,$$
		where the integral is taken with respect to the Lebesgue measure on $\Pic(X,M^\circ)^\vee_{\mathbb{R}}$, normalized by the lattice $\Pic(X,M^\circ)^\vee\subset \Pic(X,M^\circ)^\vee_{\mathbb{R}}$.
	\end{definition}
	Now we will state the main result of this article, which gives an asymptotic formula for the number of $\cM$-points of bounded height on a quasi-proper toric pair $(X,M)$. If $L$ is adjoint rigid, then we furthermore give a simple description of the leading constant. In Theorem \ref{theorem: geometric interpretation constant}, we give a geometric interpretation of this constant, which shows that it is in agreement with \cite[Conjecture 1.4]{Moe25conjecture}. Similarly, we also describe the leading constant if $L$ is toric adjoint rigid.
	\begin{theorem} \label{theorem: full asymptotics}
		Let $X$ be a proper split toric variety over $\mathbb{Q}$, let $L\in \Pic(X)_{\mathbb{Q}}$ be a big and nef $\mathbb{Q}$-divisor class and let $S$ be a positive integer. Let $(X,M)$ be a smooth toric pair which is quasi-proper with respect to $L$, and let $(\cX,\cM)$ be its toric integral model. Then there exists $\theta>0$ and a polynomial $Q$ of degree $b(\mathbb{Q}, (X,M), L)-1$ such that
		$$N_{(X,M),L,S}(B)=B^{a((X,M),L)} (Q(\log B)+O(B^{-\theta})),$$
		as $B\rightarrow \infty$.
		
		Assume that either $L$ is adjoint rigid with respect to $(X,M)$, or $S=1$ and $L$ is toric adjoint rigid with respect to $(X,M)$.
		Let $D=a_1D_1+\dots+a_nD_n$ be the unique torus-invariant $\mathbb{Q}$-divisor with $\mathbb{Q}$-linear equivalence class $a((X,M),L)L$ such that $\pr_M^*D +D_{(X,M)}$ is effective.
		The leading coefficient of $Q$ is given by
		$$C=\frac{\alpha((X,M),L)}{a((X,M),L)(b(\mathbb{Q},(X,M),L)-1)!} C_{\infty} \prod_{p\, \mathrm{prime}} C_p,$$
		where
		$$C_p=(1-p^{-1})^{\#\Gamma_{M^\circ}}\sum_{\bfm\in \fM_{\red}}p^{-a_{\bfm}}$$
		for all prime numbers $p$ not dividing $S$, and
		$$C_p=(1-p^{-1})^{\#\Gamma_{M^\circ}}\sum_{\bfm\in \mathbb{N}^n_{\red}}p^{-a_{\bfm}}$$
		for all prime numbers $p$ dividing $S$, where $a_{\bfm}=a_1 m_1+\dots+a_nm_n$.
		
		If $L$ is adjoint rigid, then $\#\Gamma_{M^\circ}=\dim(X)+b(\mathbb{Q},(X,M),L)$ and
		
		$$C_\infty=2^{\dim(X)}\sum_{\sigma\in \Sigma_{\max}} \prod_{\substack{i=1 \\ \rho_i\subset \sigma}}^n \frac{1}{a_i}.$$
		If $L$ is toric adjoint rigid, then 
		$$C_{\infty}=2^{\dim(X)} \sum_{\sigma\in \Sigma_{\overline{M^\circ},\max}}I(\sigma)C_{\infty}(\sigma),$$
		where 
		$\Sigma_{\overline{M^\circ},\max}$ is the collection of maximal cones in the fan $\Sigma_{\overline{M^\circ}}$ introduced in Notation \ref{notation: fan toric pair}, $I(\sigma)$ is the index of the subgroup $\langle [\tilde{D}_{\bfm}]  \mid \phi(\bfm)\not \in \sigma \rangle$ in $\Pic(X,M^\circ(\sigma))$ and
		$$C_{\infty}(\sigma)=\frac{\# \Pic(X,M^\circ)_{\mathrm{torsion}}}{\# \Pic(X,M^\circ(\sigma))_{\mathrm{torsion}}}\Volume(Z_\sigma)\dim(Z_\sigma)!.$$ Here $Z_\sigma$ is the polytope consisting of all linear functions $$f\colon V''=\langle [\tilde{D}_{\bfm}]\in \Pic(X,M^\circ(\sigma))\mid \bfm\in \Gamma_{M^\circ(\sigma)}\setminus \Gamma_{M^\circ} \rangle\rightarrow \mathbb{R}$$ satisfying $f([\tilde{D}_{\bfm}])\geq 0$ for all $\bfm\in \Gamma_{M^\circ(\sigma)}\setminus \Gamma_{M^\circ}$ and $f\left(\sum_{\bfm\in \Gamma_{M^\circ(\sigma)}\setminus \Gamma_{M^\circ}}[\tilde{D}_{\bfm}]\right)\leq 1$, and the volume is computed with respect to the measure $\nu''$ such that $V''/\Lambda''$ has volume $1$, where $\Lambda''$ is the image of $\Hom(\Pic(X,M^\circ(\sigma)),\mathbb{Z})$ in $V''$.
	\end{theorem}
	The proof of Theorem \ref{theorem: full asymptotics} is based on the proofs of Salberger \cite{Sal98} and de la Bretèche \cite{Bre01Toric} for Manin's conjecture for split toric varieties with the anticanonical height. 
\subsection{Examples}
	We illustrate the Theorem \ref{theorem: full asymptotics} by applying it to a few examples.
	\begin{example} \label{example: counting images worked out}
		Let $f\colon \mathbb{P}^2\dashrightarrow \mathbb{P}^2$ be the rational map given by $(x:y:z)\mapsto (x^2:y^2:xz)$ as in Example \ref{example: counting images}. We will derive the given asymptotic formula for the number of rational points of bounded height in the image. If $(\mathbb{P}^2,M)$ is the toric pair given by $\fM=\mathbb{N}^2$, then we compute that $f_*\fM$ is the monoid generated by $(2,0,0),(0,2,0),(0,0,1),(1,1,0)$ and thus $\Gamma_{f_*M}=\{(2,0,0),(0,2,0),(0,0,1),(1,1,0)\}$. The Picard group $\Pic(X,f_*M)\cong \mathbb{Z}^2\oplus \mathbb{Z}/2\mathbb{Z}$ is generated by $\tilde{D}_{(2,0,0)}, \tilde{D}_{(1,1,0)}$ together with the torsion element $\tilde{D}_{(2,0,0)}-\tilde{D}_{(0,2,0)}$. Furthermore, we have that $$D_{(\mathbb{P}^2,f_*M)}=\pr^*_{f_* M}\left(\tfrac{1}{2}D_1+\tfrac{1}{2}D_2+D_3\right),$$
		so $a((\mathbb{P}^2,f_*M),[D_1])=2$ and $b(\mathbb{Q},(\mathbb{P}^2,f_*M),[D_1])=2$. Now Theorem \ref{theorem: full asymptotics} shows that the number of $f_*\cM$-points of height at most $B$ tends to $B\left(Q(\log B)+O\left(B^{-\theta}\right)\right)$, where $Q$ is a linear polynomial with leading coefficient $$C=\frac{1/4}{2}\cdot 8\cdot \prod_{p\, \mathrm{prime}}(1-2p^{-2}+2p^{-3}-p^{-4}).$$ Now the asymptotic stated in Example \ref{example: counting images} follows by combining this result with Corollary \ref{corollary: image pair}, where we have $\Delta=2$ as $f$ has degree $2$.
	\end{example}
	\begin{example} \label{example: weak Campana points projective plane}
		In \cite[Section 3.2.1]{PSTVA21}, the weak Campana points on $(\mathbb{P}^2, \frac{1}{2} D_1+\frac{1}{2} D_2+\frac{1}{2} D_3)$ over $\mathbb{Z}$ are considered and compared to the Campana points on the same Campana pair. The set of weak Campana points is
		$$\{(x:y:z)\mid x,y,z\in \mathbb{Z}\setminus\{0\}, \, \gcd(x,y,z)=1, \, xyz \text{ is squareful}\}\subset \mathbb{P}^2(\mathbb{Q}),$$
		while the set of Campana points is the subset consisting of the points $(x:y:z)$ such that each integer $x,y,z$ is squareful.
		In \cite[Propostition 3.6]{PSTVA21} it is proven that the number of weak Campana points with Weil height at most $B$ and $xyz\neq 0$ is at least $c_1B^{3/2}\log B$ as $B\rightarrow \infty$, for some constant $c_1>0$. Using Theorem \ref{theorem: full asymptotics}, we can compute a precise asymptotic for the number of weak Campana points of bounded height.
		Let $(\mathbb{P}^2,M)$ be the pair corresponding to the weak Campana points on $(\mathbb{P}^2, \frac{1}{2} D_1+\frac{1}{2} D_2+\frac{1}{2} D_3)$. For the divisor class $L=[D_1]\in \Pic(\mathbb{P}^2)$, \cite[Lemma 4.39]{Moe25conjecture} implies that the pair $(\mathbb{P}^2,M^{\circ})$ as in Definition \ref{def: alpha-constant} is contained in the proper pair $(X,M')$ given by $\fM'=\{(0,0,0), (2,0,0), (0,2,0),(0,0,2),(1,1,0),(1,0,1),(0,1,1)\}$. Since $\Gamma_{M'}=\fM'\setminus \{(0,0,0)\}$ and
		$$\pr_{M'}^* [D_1]=\frac{1}{3}\pr^*_{M'}[D_1+D_2+D_3]=\frac{2}{3}\sum_{\bfm\in \Gamma_{M'}}[\tilde{D}_{\bfm}]=-\frac{2}{3} K_{(X,M')},$$
		the adjoint divisor of $L$ with respect to $(X,M)$ is $0$ so
		we must have $(\mathbb{P}^2,M^\circ)=(\mathbb{P}^2,M')$. Furthermore, this description of $\pr_{M'}^* [D_1]$ also implies that $a((\mathbb{P}^2,M),[D_1])=\frac{3}{2}$ and $b(\mathbb{Q},(X,M),[D_1])=\rank \Pic(X,M^\circ)$.
		Since $D_1$, $D_2$ and $D_3$ are linearly equivalent to each other, the divisor $\tilde{D}_{(1,0,1)}$ is linearly equivalent to $2\tilde{D}_{(0,2,0)}-2\tilde{D}_{(0,0,2)}+\tilde{D}_{(1,1,0)}$ and similarly $\tilde{D}_{(0,1,1)}$ is linearly equivalent to $2\tilde{D}_{(2,0,0)}-2\tilde{D}_{(0,0,2)}+\tilde{D}_{(1,1,0)}$. By Proposition \ref{prop: Picard group toric pair} these are the only relations between torus-invariant prime divisors on $(X,M^\circ)$, so $\Pic(X,M^{\circ})\cong \mathbb{Z}^4$ is freely generated by $\tilde{D}_{(2,0,0)}$, $\tilde{D}_{(0,2,0)}$, $\tilde{D}_{(0,0,2)}$ and $\tilde{D}_{(1,1,0)}$, and thus $b(\mathbb{Q},(X,M),L)=4$.
		The effective cone of $(\mathbb{P}^2,M^{\circ})$ is generated by the divisors $\tilde{D}_{(2,0,0)},\tilde{D}_{(0,2,0)},\tilde{D}_{(0,0,2)}, \tilde{D}_{(1,1,0)}, \tilde{D}_{(1,0,1)},\tilde{D}_{(0,1,1)}$ by Proposition \ref{prop: effective cone generated by torus-invariant}.
		By subdividing the dual of the effective cone $\Eff^1(X,M)$ into two simplicial cones, we compute the $\alpha$-constant using \cite[Example 2.1]{Bar93}:
		$$\frac{\alpha((\mathbb{P}^2,M),[D_1])}{a((\mathbb{P}^2,M),[D_1])(b(\mathbb{Q},(X,M),[D_1])-1)!}=\frac{1}{48}.$$
		Since $a((\mathbb{P}^2,M))D_1$ is linearly equivalent to $\frac{1}{2}\left(D_1+D_2+D_3\right)$, we find
		$C_{\infty}=4\cdot 4\cdot 3=48$.
		
		Finally, for each prime,
		$$C_p=\left(1-p^{-1}\right)^6 \left(\frac{1-p^{-3/2}}{\left(1-p^{-1/2}\right)^3} -3p^{-1/2}\right).$$
		
		We conclude that
		$$\#\left\{(x:y:z)\in \mathbb{P}^{2}(\mathbb{Q})\,\middle\vert
		\begin{aligned} &\,x,y,z\in \mathbb{Z}\setminus\{0\},\, \gcd(x,y,z)=1, \\
			&\, xyz \text{ is squareful},\, \max(|x|,|y|, |z|)\leq B
		\end{aligned}
		\, \right\}
		= B^{3/2} (Q(\log B)+O(B^{-\theta}))$$
		as $B\rightarrow \infty$, where $\theta>0$ is a constant and $Q$ is a cubic polynomial with leading coefficient
		$$\prod_{p\, \mathrm{prime}} \left(1-p^{-1}\right)^6 \left(\frac{1-p^{-3/2}}{\left(1-p^{-1/2}\right)^3} -3p^{-1/2}\right)\approx 0.862.$$
		
	\end{example}
	\begin{remark}
		In the previous example, there is a more elementary method to see that there exists a constant $c>0$ such that for any real number $B>2$ there are at least $cB^{3/2} (\log B)^3$ tuples $(x:y:z)$ with $|x|,|y|,|z|\leq B$ and $xyz$ squareful and nonzero, which we will now give. For every choice of pairwise coprime integers $n_1,n_2,n_3,n_4,n_5,n_6$, if we set $$(x:y:z):=(n_1^2 n_5 n_6: n_2^2 n_4 n_6: n_3^2 n_4 n_5),$$
		then $xyz=\prod_{i=1}^6 n_i^2$ is a squareful number. The probability that $6$ positive integers less than a given bound are pairwise coprime is at least $c'=\prod_{p\, \mathrm{prime}} (1-\tfrac{5}{p})(1-\tfrac{1}{p})^5>0$.
	Thus the number $N(B)$ of points $(x:y:z)\in \mathbb{P}^2(\mathbb{Q})$ with $xyz$ squareful and nonzero and furthermore $\max(|x|,|y|,|z|)\leq B$ is at least $$c'\#\{(n_1,n_2,n_3,n_4,n_5,n_6)\in (\mathbb{N}^*)^6\mid \max(n_1^2 n_5n_6, n_2^2 n_4 n_6, n_3^2 n_4 n_5) \leq B\}.$$
	The integer $N(B)$ is equal to the integral $c' \idotsint_{\bfx \in A(B)} \d x_1 \dots \d x_6$, where $A(B)$ is the set of all $(x_1,\dots,x_6)\in [1,\infty)^6$ such that $\max(\lfloor x_1\rfloor ^2 \lfloor x_5\rfloor \lfloor x_6\rfloor, \lfloor x_2\rfloor ^2 \lfloor x_4\rfloor \lfloor x_6\rfloor, \lfloor x_3\rfloor^2 \lfloor x_4\rfloor \lfloor x_5\rfloor)\leq B$. By using the trivial upper bound $\lfloor x_i \rfloor\leq x_i$ for all $i\in \{1,2,3,4,5,6\}$ together with the change of variables $y_i=\log x_i$, $N(B)$ is at least
	$$c' \idotsint_{\mathbf{y} \in A'(\log B)} e^{y_1+\dots +y_6} \d y_1\dots \d y_6,$$
	where $A'(\log B)$ is the set of all $(x_1,\dots,x_6)\in [0,\infty)^6$ such that $\max(2y_1+ y_5+y_6, 2y_2+ y_4+ y_6, 2y_3+ y_4+ y_5)\leq \log B$. On the domain $A'(B)$ the maximum value attained by the integrand is $B^{3/2}$, which is attained on a three dimensional face $F(\log B)$ of $A'(\log B)$. Let $\mathbf{n}_1,\dots, \mathbf{n}_3\in A'(\log B)$ be vectors generating the normal space of $F(\log B)$ with respect to the standard inner product on $\mathbb{R}^6$. Let $F'(1)$ be an open of $F(1)\cap A'(1)$ which is bounded away from the boundary of $A'(1)$. Then there exists $\epsilon>0$ such that $V(1):=\{\mathbf{y}-a_1\mathbf{n}_1-a_2\mathbf{n}_2-a_3\mathbf{n}_3\in \mathbb{R}^6 \mid \mathbf{y}\in F'(1), a_1,a_2,a_3\in [0,\epsilon)\}$ is fully contained in $A'(1)$. Let $V(\log B)=\log B V(1)$. By the previous lower bound for $N(B)$, we see that
	$$N(B)\geq c' \idotsint_{\mathbf{y} \in V(\log B)} e^{y_1+\dots +y_6} \d y_1\dots \d y_6.$$
	The integral over $V(\log B)$ is equal to
	$$\iiint_{\mathbf{y} \in F'(\log B)}\iiint_{a_1,a_2,a_3\in [0, \epsilon]} B^{3/2}e^{\sum_{i=1}^6(-a_1n_{1,i}-a_2 n_{2,i}-a_3 n_{3,i})}\d a_1 \d a_2 \d a_3 \d \mu,$$
	where $\mu$ is the Lebesgue measure on $F'(\log B)$. This is in turn equal to $$\Volume(F'(\log B))B^{3/2}(1+o(1))=\Volume(F'(1))B^{3/2}(\log B)^3(1+o(1))$$ as $B\rightarrow \infty$, which gives the desired lower bound for $N(B)$.
\end{remark}
Using the description of the $b$-invariant for weak Campana points given in \cite[Theorem 1.7]{Moe25conjecture}, we can also determine the asymptotic growth for the number of points on projective space for which the product of the coordinates is an $m$-full number, generalizing Example \ref{example: weak Campana points projective plane}.
\begin{example} \label{example: weak Campana projective space points with equal weights}
	Let $m$ and $n$ be positive integers. We now explain why Theorem \ref{theorem: asymptotics proper toric} implies that 
	\begin{multline*}
		N(B)=\#\left\{(x_1:\dots: x_n)\in \mathbb{P}^{n-1}(\mathbb{Q})\middle\vert
		\begin{aligned} &\,x_1,\dots, x_n\in \mathbb{Z}\setminus \{0\},\, \gcd(x_1,\dots,x_n)=1, \\
			&\, \prod_{i=1}^n x_i \text{ is }m\text{-full},\, \max(|x_1|,\dots, |x_n|)\leq B
		\end{aligned}
		\, \right\}=\\ B^{n/m}(Q(\log B)+O(B^{-\theta}))
	\end{multline*}
	as $B\rightarrow \infty$,
	where $\theta>0$ is a constant and $Q$ is a polynomial of degree
	$$\binom{m+n-1}{n-1}-\binom{m-1}{n-1}-n.$$
	Note that $N(B)$ is the number of weak Campana points on the Campana pair $\left(\mathbb{P}^{n-1}, \sum_{i=1}^n \left(1-\frac{1}{m}\right)D_i\right)$ of height at most $B$, where the divisors $D_1,\dots,D_n$ are the coordinate hyperplanes. In particular, the log-anticanonical divisor class is given by $\sum_{i=1}^n \frac{1}{m}[D_i]$ and \cite[Theorem 1.7]{Moe25conjecture} implies that the Fujita invariant is $n/m$, while the $b$-invariant is equal to
	$$-n+\#\{(a_1,\dots,a_n)\in \mathbb{N}^n\mid a_1+\dots+a_n=m, \, \min(a_1,\dots,a_n)=0\}.$$
	The number of ways to write $m$ as a sum of $n$ (nonzero) integers is $\binom{m+n-1}{n-1}$ (respectively $\binom{m-1}{n-1}$), which gives the expression for the degree of $Q$. 
\end{example}

\subsection{Geometric interpretation of the constant}
Before giving the proof of Theorem \ref{theorem: full asymptotics}, we give a geometric interpretation for the leading constant obtained in the theorem in the case where $L$ is adjoint rigid with respect to $(X,M)$, by interpreting the constant as an adelic integral. In particular, we show that the leading constant in Theorem \ref{theorem: full asymptotics} agrees with \cite[Conjecture 1.4]{Moe25conjecture}.

\begin{theorem} \label{theorem: geometric interpretation constant}
	In the setting of Theorem \ref{theorem: full asymptotics}, assume that $L$ is adjoint rigid with respect to $(X,M)$. Then the leading coefficient of the polynomial $Q$ is given by
	
	$$C=\frac{\alpha((X,M),L)}{a((X,M),L) (b(\mathbb{Q},(X,M),L)-1)!}\int_{x\in (\cX,\cM)(\mathbf{A}_{\mathbb{Z}[1/S]})} \frac{1}{H_{a((X,M),L)L+K_X}(x)} \d \tau_{(X,M^\circ)},$$
	where $H_{a((X,M),L)L+K_X}$ is the toric height corresponding to the $\mathbb{Q}$-divisor class $a((X,M),L)L+K_X$ as defined in \cite[Definition 2.1.7]{BaTs95}, and the measure $\tau_{(X,M^\circ)}$ is the measure from \cite[\S 8]{Moe25conjecture} defined using the toric metric on the canonical divisor class $K_X$ as in \cite[Theorem 2.1.6]{BaTs95}. Furthermore, this agrees with \cite[Conjecture 1.4]{Moe25conjecture}.
	
	In particular, if $(X,M)$ is a pair corresponding to Campana points, then the constant is compatible with the prediction in \cite[\S 3.3]{PSTVA21}.
\end{theorem}
\begin{remark}
	Strictly speaking, Conjecture \cite[Conjecture 1.4]{Moe25conjecture} only applies to proper pairs. However, every term in the conjecture straightforwardly generalizes to quasi-proper pairs, so we can still compare Theorem \ref{theorem: full asymptotics} with the conjecture for such pairs in general.
\end{remark}
\begin{proof}
	We will prove the theorem by computing the the Tamagawa constant as a product over all places over $\mathbb{Q}$, which will show that it agrees with the product $C_\infty\prod_{p\, \mathrm{prime}}C_p$ in Theorem \ref{theorem: full asymptotics}.
	First we note that $\Pic(X_{\overline{\mathbb{Q}}}, M^{\circ})=\Pic(X, M^{\circ})$, so
	$$L(s,\Pic(X_{\overline{\mathbb{Q}}},M^\circ)/\{\mathrm{torsion}\})=\zeta(s)^{b((X,M),L)},$$
	where $\zeta$ is the Riemann zeta function.
	This implies that the convergence factors are simply given by $\lambda_p^{-1}=(1-p^{-1})^{b(\mathbb{Q},(X,M),L)}$ and $L^*(1,\Pic(X_{\overline{Q}},M^\circ)=1$. In particular, $\tau_{(X,M^\circ)}=\tau_{X,\infty}\times \prod_{p \, \mathrm{prime}} (1-p^{-1})^{b((X,M),L)} \tau_{X,p}$, where the measures $\tau_{X,\infty}$ and $\tau_{X,p}$ are the local measures on $X(\mathbb{R})$ and $X(\mathbb{Q}_p)$ as in \cite[\S 2.1.8]{CLTs10} induced by the toric metric on the canonical divisor class $K_X$ as defined in \cite[Theorem 2.1.6]{BaTs95}.
	The $\mathbb{Q}$-divisor class $a((X,M),L)L+K_X$ on $X$ is represented by the $\mathbb{Q}$-divisor class $a((X,M),L)D+D_X$, where $D$ is as in Theorem \ref{theorem: full asymptotics} and $D_X=-\sum_{i=1}^n D_i$.
	For every prime number $p$ not dividing $S$, we write $\delta_{M,p}$ for the indicator function of $p$-adic $\cM$-points on $(\cX,\cM)$. For a prime number $p$ dividing $S$, we instead set $\delta_{M,p}=1$. 
	The adelic integral is equal to the product
	$$\int_{X(\mathbb{R})} \frac{1}{H_{a((X,M),L)D+D_X,\infty}(x)}\d \tau_{X,\infty}\prod_{p \, \mathrm{prime}} \int_{X(\mathbb{Q}_p)} \frac{(1-p^{-1})^{b(\mathbb{Q},(X,M),L)}\delta_{M,p}(x)}{H_{a((X,M),L)D+D_X,p}(x)}\d \tau_{X,p},$$
	where $H_{a((X,M),L)L+K_X}=\prod_{v\in \Omega_{\mathbb{Q}}}H_{a((X,M),L)D+D_X,v}$ as in \cite[Definition 2.1.5]{BaTs95}. We will now verify that $C_\infty$ is equal to the real integral and that $C_p$ is equal to the $p$-adic integral, for each prime number $p$. 
	We will first start with the archimedean place $\infty$.
	
	For a maximal cone $\sigma\in \Sigma$, let $C_{\sigma}(\mathbb{R})\subset X(\mathbb{R})$ be the subset as defined in \cite[Notation 9.1]{Sal98}. As shown in the proof of \cite[Lemma 9.10]{Sal98}, $C_{\sigma}(\mathbb{R})$ is the set of all $(x_1:\dots:x_n)\in X(\mathbb{R})$ such that $|x_1|,\dots,|x_n|\leq 1$ and $x_i=1$ for all $i\in \{1,\dots,n\}$ with $\rho_i\not \subset \sigma$. In particular we can identify $C_{\sigma}(\mathbb{R})$ with $[0,1]^d$. Under this identification, the measure $\tau_{X,\infty}$ corresponds to the Lebesgue measure on $[0,1]^d$ as shown in the proof of \cite[Proposition 9.16]{Sal98}. By construction, we have $X(\mathbb{R})= \bigcup_{\sigma\in \Sigma_{\max}} C_{\sigma}(\mathbb{R})$ and the proof of \cite[Proposition 9.16]{Sal98} implies that $C_{\sigma}(\mathbb{R})\cap C_{\sigma'}(\mathbb{R})$ for any two maximal cones $\sigma\neq \sigma'$. Finally, for $P=(x_1:\dots: x_n)\in C_{\sigma}(\mathbb{R})$, the height is simply given $H_{a((X,M),L)D+D_X,\infty}(P)= \prod_{\substack{i=1 \\ \rho_i\subset \sigma}}^n x_i^{1-a_i}$.
	Thus we obtain
	$$\int_{x\in X(\mathbb{R})} \frac{1}{H_{a((X,M),L)D+D_X,\infty}(x)}\d \tau_{X}=2^{\dim(X)}\sum_{\sigma\in\Sigma_{\max}} \prod_{\substack{i=1 \\ \rho_i\subset \sigma}}^n \int_{0}^1 x_i^{a_i-1}\d x_i=C_\infty.$$
	
	Let $p$ be a prime. By \cite[Proposition 9.14]{Sal98}, the measure $\tau_{X,p}$ restricts to the Haar measure on $\mathbb{A}^d(\mathbb{Z}_p)=\mathbb{Z}_p^d$ for all toric subschemes $\mathbb{A}^d\subset \cX$. In particular it follows that for all $\bfm\in \fM_{\red}$, the set
	$$V_\bfm=\{x\in X(\mathbb{Q}_p)\mid \mult_p(x)=\bfm\}$$
	has volume equal to $(1-p^{-1})^{\dim(X)}p^{-\sum_{i=1}^n m_i}$. Additionally, the function $H_{a((X,M),L)D+D_X,p}$ is constant on $V_\bfm$, where it takes the value $p^{a_\bfm-\sum_{i=1}^n m_i}$.
	This implies
	$$\int_{X(\mathbb{Q}_p)} \frac{(1-p^{-1})^{b(\mathbb{Q},(X,M),L)}\delta_{M,p}(x)}{H_{a((X,M),L)D+D_X,p}(x)}\d \tau_{X,p}=(1-p^{-1})^{\dim(X)+b(\mathbb{Q},(X,M),L)}\sum_{\bfm\in \fM_{\red}}p^{-a_\bfm}$$
	if $p$ does not divide $S$, and
	$$\int_{X(\mathbb{Q}_p)} \frac{(1-p^{-1})^{b(\mathbb{Q},(X,M),L)}}{H_{a((X,M),L)D+D_X,p}(x)}\d \tau_{X,p}=(1-p^{-1})^{\dim(X)+b(\mathbb{Q},(X,M),L)}\sum_{\bfm\in \mathbb{N}^n_{\red}}p^{-a_\bfm}$$
	if $p$ divides $S$. In either case, this agrees with the factor $C_p$ from Theorem \ref{theorem: full asymptotics}. Thus we obtain the desired identity.
	
	The compatibility with the conjecture in \cite{PSTVA21} follows from this identity together with the straightforward identity
	$$\frac{\d \tau_{U',D_{\bfm}}}{H_{a((X,M),L)L+K_X+D_{\bfm}}}= \frac{\d \tau_{U'}}{H_{a((X,M),L)L+K_X}},$$
	where $U'$ is the complement of the support of $a((X,M),L)L+K_X+D_{\bfm}$,
	$\tau_{U'}$ is the Tamagawa measure as in \cite[Definition 2.8]{CLTs10} and $\tau_{U',D_{\bfm}}= \frac{\tau_{U'}}{H_{D_{\bfm}}}$.
	
	Finally, we verify that the leading constant agrees with the prediction in \cite[Conjecture 1.4]{Moe25conjecture}. Note that the adelic integral $\int_{x\in (\cX,\cM)(\mathbf{A}_{\mathbb{Z}[1/S]})} \frac{1}{H_{a((X,M),L)L+K_X}(x)} \d \tau_{(X,M^\circ)}$ is equal to $\hat{\tau}(1)$ from \cite[Definition 8.3]{Moe25conjecture}, where $1$ is the trival Brauer class. We will show that $\hat{\tau}(b)=0$ for every nontrivial Brauer class $b\in \Br_1(X,M)/\Br K$.
	
	Using \cite[Lemma 7.8]{Moe25conjecture}, we can identify $\Br_1(X,M^\circ)/\Br \mathbb{Q}$ with the group of automorphic characters $\chi$ on $U\cong\mathbb{G}_m^d$ such that the induced character $(\chi_{\bfm})_{\bfm\in \Gamma_{M^\circ}}$ on $\mathbb{G}_m^{\Gamma_{M^\circ}}$ is identically $1$. For a place $v$, let $\cK_v$ be the maximal compact subgroup of $\mathbb{G}_m(K_v)^d$: i.e.\ for each prime number $p$ we have $\cK_p=\mathbb{G}_m(\mathbb{Z}_p)^d$ and $\cK_{\infty}= \{\pm 1\}^d\subset \mathbb{G}_m(\mathbb{R})^d$. For any place $v\in \Omega_{\mathbb{Q}}$ given element $x\in (\cX,\cM)(\cO_v)$ and any element $t\in \cK_v$, $tx$ also lies in $(\cX,\cM)(\cK_v)$ and $$H_{v,a((X,M),L)L+K_X}(tx)=H_{v,a((X,M),L)L+K_X}(x).$$ 
	
	In particular it follows that for every $\chi\in \Br_1(X,M^\circ)/\Br \mathbb{Q}$ and every $t\in \cK_v$ we have
	$$\int_{x\in (\cX,\cM)(\cO_v)} \frac{\chi(x)\d_{X,M^\circ}}{H_{v,a((X,M),L)L+K_X}(x)}=\int_{x\in (\cX,\cM)(\cO_v)} \frac{\chi(t)\chi(x)\d_{X,M^\circ}}{H_{v,a((X,M),L)L+K_X}(x)},$$
	which implies that $\hat{\tau}(\chi)$ is zero unless $\chi$ is identically $1$ on where $\cK=\prod_{v\in \Omega_{\mathbb{Q}}} \cK_v$. Therefore, in order to determine the Tamagawa constant, we only need to consider characters in the group $B$, where $B\subset \Br_1(X,M)/\Br \mathbb{Q}$ is the subgroup of characters which are trivial on $\cK$.
	By class field theory \cite[Ch. VI, \S6; Ch. VII, \S6]{Neu99}, the group of automorphic characters on $\mathbb{G}_m$ which are both torsion and trivial on the maximal compact subgroup is isomorphic to the class group of $\mathbb{Q}$. Thus as the class group of $\mathbb{Q}$ is trivial, $B$ only contains the trivial character, proving that only the trivial character contributes to the leading constant in \cite[Conjecture 1.4]{Moe25conjecture}.
\end{proof}

\subsection{Proof of Theorem \ref{theorem: full asymptotics}, part I: upper bounds}
This section is devoted to proving the following weak form of Theorem \ref{theorem: full asymptotics}.
\begin{lemma} \label{lemma: weak form asymptotics}
	Let $(X,M)$ be a smooth toric pair over $\mathbb{Q}$ which is quasi-proper with respect to a big and nef $\mathbb{Q}$-divisor class $L$ on $X$. Then there exists a constant $\theta>0$ and a polynomial $Q$ of degree at most $b(\mathbb{Q},(X,M),L)-1$ such that
	$$N_{(X,M),L,S}(B)= B^{a((X,M),L)} (Q(\log B)+ O(B^{-\theta})).$$
\end{lemma}

The proof of this lemma is based on the following Tauberian theorem by de la Bretèche.
\begin{theorem}\cite[Théorème 1]{Bre01Sum} \label{theorem: Breteche 1}
	Let $f\colon \mathbb{N}^k\rightarrow \mathbb{R}$ be a nonnegative arithmetic function and let $F$ be the associated Dirichlet series
	$$F(\bfs)=\sum_{r_1=1}^\infty\dots \sum_{r_k=1}^\infty \frac{f(r_1,\dots,r_k)}{r_1^{s_1}\dots r_k^{s_k}}.$$
	Suppose there exists $\bfalpha\in \mathbb{R}_{> 0}^k$ such that $F$ satisfies the following three properties:
	\begin{enumerate}
		\item [(P1)] For all $\bfs$ with $\Re(\bfs)>\bfalpha$, the series $F(\bfs)$ converges absolutely.
		\item [(P2)] There exist finite collections $\mathscr{L}=\{l_1,\dots,l_{\tilde{n}}\}$ and $\mathscr{R}$ of linear forms with nonnegative coefficients such that the function $H$ defined by
		$$H(\bfs)=F(\bfs+\bfalpha)\prod_{i=1}^{\tilde{n}}l_i(\bfs)$$
		can be analytically continued to a holomorphic function defined on
		$$\mathscr{D}(\delta_1)=\{\bfs\in \mathbb{C}^k\mid \Re(l(\bfs))>-\delta_1, \forall l\in \mathscr{L}\cup \mathscr{R}\},$$
		for some positive constant $\delta_1$.
		\item [(P3)] There exists $\delta_2>0$ such that for all $\epsilon>0,\epsilon'>0$, the upper bound
		$$|H(\bfs)|\ll (1+\lVert\Im(\bfs)\rVert_1^\epsilon)\prod_{i=1}^{\tilde{n}} (1+|\Im(l_i(\bfs))|)^{1-\delta_2 \min\{0, \Re(l_i(\bfs))\}}$$
		is uniform in the domain $\mathscr{D}(\delta_1-\epsilon')\cap \{\bfs\in \mathbb{C}^k| \Re(\bfs)< (1,\dots,1)\}$.
	\end{enumerate}
	Then there exists a polynomial $Q$ of degree at most $\tilde{n}-\rank(l_1,\dots,l_{\tilde{n}})$ such that $$S(B)=\sum_{r_1=1}^B\dots\sum_{r_k=1}^B f(r_1,\dots,r_k)=B^{\sum_{i=1}^k \alpha_i}(Q(\log B)+O(B^{-\theta}))$$
	for all $B\geq 1$,
	where $\theta>0$ is a constant depending on $\bfalpha,\delta_1,\delta_2,\mathscr{L}, \mathscr{R}$.
\end{theorem}

\begin{remark}
	Theorem \ref{theorem: Breteche 1} has a typo in its original formulation \cite[Théorème 1]{Bre01Sum}, as it requires the the upper bound in (P3) to be satisfied on the whole of $\mathscr{D}(\delta_1-\epsilon')$. This is a stronger assumption than necessary and desired, as it would imply that $H$ is bounded on $\mathbb{R}_{>0}^k$, which would exclude the original application of the theorem in \cite{Bre01Toric}.
\end{remark}
\subsubsection{Dirichlet series}
Since the height function satisfies $H_{tL}(\bfx)=H_{L}(\bfx)^t$ for any $t\in \mathbb{Q}_{>0}$, we assume without loss of generality that $L\in \Pic(X)$.
As $L$ is integral, the function $\bfx^{L(\sigma_j)}$ is simply a monomial, and thus takes integer values on integer inputs.


For $(r_1,\dots,r_k)\in \mathbb{N}^k$, we define
\begin{equation*}
	f(r_1,\dots, r_k)=\#\left\{\bfd\in (\mathbb{N}^*)^n\middle\vert \begin{aligned}
		&\, \bfd^{L(\sigma_j)}=r_j \text{ for all }j=1,\dots, k; \quad \gcd(\bfd^{\hat{\sigma}_1},\dots, \bfd^{\hat{\sigma}_k})=1;\\ &\, \mult_p(\bfd)\in \fM\text{ for all prime numbers }p\nmid S\end{aligned}
	\right\},  
\end{equation*}
where $\mult_p(\bfd)=(v_p(x_1),\dots, v_p(x_n))$ is the tuple given by the $p$-adic valuations of the components of $\bfd$, and the monomials $$\bfx^{\hat{\sigma}}=\prod_{\substack{i=1 \\ \rho_i\not\subset \sigma}}^n x_i$$ are defined as in Section \ref{section: Cox coordinates}.

In this notation, $N_{(X,M),L,S}(B)=2^{\dim(X)} S(B)$, where $$S(B):=\sum_{r_1=1}^B\dots\sum_{r_k=1}^B f(r_1,\dots,r_k).$$
Here the factor of $2^{\dim(X)}$ accounts for the fact that $f$ only counts the points that can be described using positive coordinates.
We estimate the sum $S(B)$ by considering the multiple Dirichlet series
$$F(\bfs):=\sum_{r_1=1}^\infty \dots \sum_{r_k=1}^\infty \ddfrac{f(r_1,\dots r_k)}{r_1^{s_1}\dots r_k^{s_k}}.$$

In order to apply Theorem \ref{theorem: Breteche 1}, we will rewrite $F$ as a multiple Dirichlet series of a multiplicative function. 
Let $\chi\colon (\mathbb{N}^*)^n\rightarrow \{0,1\}$ be the characteristic function of the set
$$\{\bfd\in (\mathbb{N}^*)^n\mid \gcd(\bfd^{\hat{\sigma}_1},\dots, \bfd^{\hat{\sigma}_k})=1,\mult_p(\bfd)\in \fM \text{ for all primes } p\nmid S\},$$
so
$$\frac{f(r_1,\dots r_k)}{r_1^{s_1}\dots r_k^{s_k}}=\sum_{\substack{\bfd\in \mathbb{N}^n \\ \bfd^{L(\sigma_j)}=r_j \, \forall j\in \{1,\dots,k\}}} \frac{\chi(\bfd)}{d_1^{l^{(1)}(\bfs)}\dots d_n^{l^{(n)}(\bfs)}}.$$
This equality implies
\begin{align*}
	F(\bfs)=\sum_{d_1=1}^\infty \dots \sum_{d_n=1}^\infty \ddfrac{\chi(\bfd)}{d_1^{l^{(1)}(\bfs)}\dots d_n^{l^{(n)}(\bfs)}}.
\end{align*}
Since the condition $\gcd(\bfd^{\hat{\sigma}_1},\dots, \bfd^{\hat{\sigma}_k})=1$ only depends on the valuations $\mult_p(\bfd)$ at each prime, the function $\chi$ is multiplicative in the sense that $$\chi(\bfd\bfd')=\chi(\bfd)\chi(\bfd')$$ for all $\bfd,\bfd'\in (\mathbb{N}^*)^n$ satisfying $\gcd(d_1\dots d_n,d_1'\dots d'_n)=1$.

Since $\chi$ is multiplicative, we can write
$$F(\bfs)=\prod_{p\, \mathrm{prime}} F_p(\bfs),$$
where for primes numbers $p$ not dividing $S$,
$$F_p(\bfs)=\sum_{m_1=1}^\infty\dots \sum_{m_n=1}^\infty \ddfrac{\chi(p^{m_1},\dots, p^{m_n})}{p^{l_{\bfm}(\bfs)}}= \sum_{\bfm\in \fM_{\red}} p^{-l_{\bfm}(\bfs)},$$
and similarly for prime numbers $p$ dividing $S$,
$$F_p(\bfs)= \sum_{\bfm\in \mathbb{N}^n_{\red}} p^{-l_{\bfm}(\bfs)}.$$
Here we write $l_{\bfm}=m_1l^{(1)}+\dots +m_n l^{(n)}$ for $\bfm\in \mathbb{N}^n$. Note that for all $i\in \{1,\dots, n\}$ and $\bfs\in \mathbb{Q}^k$, the value $l^{(i)}(\bfs)$ is the coefficient of $D_i$ in $\sum_{j=1}^k s_j L(\sigma_j)$. This can be seen as a direct consequence of the definition of the linear forms $l^{(1)},\dots, l^{(n)}$ as given in \eqref{Equation: l^{(i)}}. For $\bfm\in \Gamma_M$, this gives a simple geometric interpretation for the linear form $l_{\bfm}$: for $\bfs\in \mathbb{Q}^k$, the value $l_{\bfm}(\bfs)$ is the coefficient of $\tilde{D}_{\bfm}$ in $\pr^*_M \sum_{j=1}^k s_j L(\sigma_j)$.

Using the product formula $F(\bfs)=\prod_{p\, \mathrm{prime}} F_p(\bfs)$, we determine an open set on which $F(\bfs)$ converges.
\begin{proposition} \label{prop: convergence}
	The series $F_p(\bfs)$ converges absolutely in the region $$V=\{\bfs\in \mathbb{C}^k\mid \Re(l^{(i)}(\bfs))>0 \text{ for all } i=1,\dots, n\}.$$ Furthermore, for any $\epsilon>0$ and any prime number $p\not\in S$, $$F_p(\bfs)=1+O(p^{-1-\epsilon})$$ in the region $V\cap \{\bfs\in \mathbb{C}^k\mid \Re(l_{\mathbf{m}}(\bfs))>1+\epsilon \text{ for all } \mathbf{m}\in \fM\setminus \{\boldsymbol{0}\}\}$. Here the implicit constant depends on $\epsilon$ but not on $p$.
	
	Consequently, the series $F(\bfs)$ converges in the region $$V\cap \{\bfs\in \mathbb{C}^k\mid \Re(l_{\mathbf{m}}(\bfs))>1 \text{ for all } \mathbf{m}\in \fM\setminus \{\boldsymbol{0}\}\}.$$
\end{proposition}
\begin{proof}
	The series $$\sum_{\bfm\in \mathbb{N}^n} p^{-\Re(l_{\bfm}(\bfs))}= \prod_{i=1}^n \sum_{m_i=1}^\infty \left(p^{-\Re(l^{(i)}(\bfs))}\right)^{m_i}$$ converges to $\prod_{i=1}^n \left(1-p^{-\Re(l^{(i)}(\bfs))}\right)^{-1}$ for all $\bfs\in V$, as it is simply a product of convergent geometric series. By comparing $F_p(\bfs)$ with this series we directly obtain the absolute convergence of $F_p(\bfs)$ for $\bfs\in V$. Let $\bfm_1,\dots, \bfm_t$ be the minimal nonzero elements in $\fM$. Every nonzero element $\bfm\in \fM$ can be written as $\bfm_i+\bfm'$ for some $i\in \{1,\dots,t\}$ and $\bfm'\in \mathbb{N}^n$, so $|F_p(\bfs)-1|$ is dominated by the series $$\left(p^{-\Re(l_{\bfm_1}(\bfs))}+\dots+p^{-\Re(l_{\bfm_t}(\bfs))} \right) \sum_{\bfm\in \mathbb{N}^n} p^{-\Re(l_{\bfm}(\bfs))}.$$ As $\sum_{\bfm\in \mathbb{N}^n} p^{-\Re(l_{\bfm}(\bfs))}$ is bounded on $V\cap \{\bfs\in \mathbb{C}^k\mid \Re(l_{\mathbf{m}}(\bfs))>1 \text{ for all } \mathbf{m}\in \fM\setminus \{\boldsymbol{0}\}\}$, $F_p(\bfs)$ satisfies the estimate $F_p(\bfs)=1+O(p^{-1-\epsilon})$ on this region.
	In turn, the estimate for $p\not \in S$ implies that the product $F(\bfs)=\prod_{p\, \mathrm{prime}} F_p (\bfs)$ converges whenever $\bfs$ lies in 
	
	\begin{equation*}
		V\cap \{\bfs\in \mathbb{C}^k\mid \Re(l_{\mathbf{m}}(\bfs))>1 \text{ for all } \mathbf{m}\in \fM\setminus \{\boldsymbol{0}\}\}. \tag*{\qedhere}
\end{equation*}\end{proof}

By Proposition \ref{prop: convergence}, if a tuple $\bfalpha\in \mathbb{R}_{\geq 0}^k$ satisfies $\Re(l_{\mathbf{m}}(\bfalpha))\geq 1$ for all $\mathbf{m}\in \fM\setminus \{\boldsymbol{0}\}\}$ as well as $\Re(l^{(i)}(\bfalpha))>0$ for all $i=1,\dots, n$, then condition (P1) from Theorem \ref{theorem: Breteche 1} will be satisfied. However, in order to find a good bound for $N_{(X,M),L,S}(B)$ using Theorem \ref{theorem: Breteche 1}, we need to minimize the sum $\sum_{j=1}^k \alpha_j$.

\subsubsection{Choice of \texorpdfstring{$\bfalpha$}{α}}
Let $\mathscr{P}_M$ be the following linear program: minimize the function
$\sum_{j=1}^k \alpha_j,$
under the conditions $\alpha_j\geq 0$ for all $j=1,\dots,k$ and $l_{\bfm}(\bfalpha)\geq 1$ for all $\bfm\in \fM$. Since $l_{\bfm}=m_1l^{(1)}+\dots+ m_n l^{(n)}$, the latter condition is equivalent to the condition $$\sum_{j=1}^k \alpha_j \pr_M^*L(\sigma_j)+D_{(X,M)}\geq 0.$$
This condition in turn implies that $(\sum_{j=1}^k \alpha_j)\pr_M^* L+K_{(X,M)}\in \Eff^1(X,M)$, so any solution $\bfalpha$ to $\mathscr{P}_M$ has to satisfy $\sum_{j=1}^k \alpha_j\geq a((X,M),L)$.

We will use the following proposition to show that the equality
$\sum_{j=1}^k \alpha_j= a((X,M),L)$
can be achieved.
\begin{proposition} \label{prop: optimal choice alpha}
	Let $D=a_1D_1+\dots+a_n D_n$ be a torus-invariant $\mathbb{Q}$-divisor representing $L\in \Pic(X)_{\mathbb{Q}}$.
	Then
	$$\sum_{\sigma\in \Sigma_{\max}} \left(\prod_{\rho_i\not\in\sigma} a_i\right)\mu_D(\sigma)=0,$$
	where we recall the notation $\mu_D=a_1\mu_{D_1}+\dots+ a_n\mu_{D_n}$. Thus $$\sum_{\sigma\in \Sigma_{\max}} L(\sigma)\prod_{\rho_i\not\in\sigma} a_i=\left(\sum_{\sigma\in \Sigma_{\max}} \prod_{\rho_i\not\in\sigma} a_i\right) D.$$
\end{proposition}
\begin{proof}
	Let $\sigma\in \Sigma$ be a maximal cone in the fan of $X$, and order the rays such that $\rho_1,\dots \rho_d$ are the rays in $\sigma$, where $d=\dim X$. Then $n_{\rho_1},\dots,n_{\rho_d}$ forms a lattice basis of $N$, since $X$ is smooth. Denote the corresponding dual basis by $n_{\rho_1}^*,\dots, n_{\rho_d}^*\in N^\vee$. Then we see by definition that $$\mu_D(\sigma)=\sum_{i=1}^d a_i n_{\rho_i}^*.$$
	Recall that a linear form $N\rightarrow \mathbb{Z}$ is irreducible if it is not a positive multiple of another linear form.
	For $\tau$ a facet of $\sigma$, we write $u_{\sigma,\tau}\in N^\vee$ for the unique irreducible linear form which is zero on $\tau$ and positive on $\sigma\setminus\tau$. This linear form is simply given as $u_{\sigma,\tau}=n^*_{\rho_i}$, where $\rho_i$ is the unique ray in $\sigma(1)\setminus \tau(1)$. We also write $a_{\sigma,\tau}=a_i$. By \cite[Lemma 8.9(i)]{Sal98} there is a unique maximal cone $\sigma'\neq \sigma$ containing $\tau$, which satisfies $\tau=\sigma'\cap \sigma$. Since $u_{\sigma',\tau}$ is also irreducible and zero on $\tau$, we see $u_{\sigma',\tau}=\pm n_{\rho_i}^*$. Since there are $n_1\in \sigma\setminus \tau$, $n_2\in \sigma'\setminus \tau$ such that $n_1+n_2\in \tau$, we see $u_{\sigma',\tau}=-n_{\rho_i}^*=-u_{\sigma,\tau}$.
	
	Therefore, we get 
\begin{align*}
		\sum_{\sigma\in \Sigma_{max}} \left(\prod_{\substack{i=1 \\ \rho_i\not\in\sigma}}^n a_i\right)\mu_D(\sigma)&
=\sum_{\sigma\in \Sigma_{max}} \left(\prod_{\substack{i=1 \\ \rho_i\not\in\sigma}}^n a_i\right) \sum_{\tau \text{ facet of }\sigma} a_{\sigma,\tau} u_{\sigma,\tau}\\&=\sum_{\tau \text{ facet of }\Sigma} \sum_{\substack{\sigma\in \Sigma_{\max} \\ \tau\subset\sigma}} \left(\prod_{\substack{i=1 \\ \rho_i\not\in\sigma}}a_i\right) u_{\sigma,\tau}=0.\tag*{\qedhere}
\end{align*}

\end{proof}

\begin{corollary} \label{corollary: optimal choice alpha 2}
	The $\mathbb{Q}$-divisor class $a((X,M),L)L$ can be represented by a $\mathbb{Q}$-divisor $D=a_1D_1+\dots+a_nD_n$ on $X$ such that $a_1,\dots,a_n>0$ and such that $a((X,M),L)\pr^*_M D+D_{(X,M)}$ is effective. Furthermore, for any such $D$, there exists a vector $\bfalpha\in \mathbb{R}_{>0}^k$ such that $D=\sum_{j=1}^k \alpha_j L(\sigma_j)$ and $\sum_{j=1}^k \alpha_j=a((X,M),L)$.
\end{corollary}
\begin{proof}
	Let $(X,\overline{M})$ be a proper pair as in Definition \ref{def: quasi-proper pair}.
	Let $D=a_1D_1+\dots+a_nD_n$ be a $\mathbb{Q}$-divisor such that $a((X,\overline{M}),L)\pr^*_{\overline{M}} D+D_{(X,\overline{M})}$ is an effective $\mathbb{Q}$-divisor. As $(X,\overline{M})$ is proper, this implies $a_1,\dots, a_n>0$ as $D_{(X,\overline{M})}=-\sum_{\bfm\in \Gamma_{\overline{M}}} \tilde{D}_{\bfm}$ is a divisor representing the canonical class $K_{(X,\overline{M})}$.
	Since $a((X,\overline{M}),L)=a((X,M),L)$, the divisor $a((X,M),L)\pr^*_M D+D_{(X,M)}$ is simply the restriction of $a((X,\overline{M}),L)\pr^*_{\overline{M}} D+D_{(X,\overline{M})}$ to $\Div(X,M)_{\mathbb{Q}}$.
	Now Proposition \ref{prop: optimal choice alpha} implies that if we take $$\beta_j= \frac{\prod_{\rho_i\not \in \sigma_j} a_i}{\sum_{\sigma\in \Sigma_{\max}} \prod_{\rho_i\not\in\sigma} a_i},$$ then $D=\sum_{j=1}^k \beta_j L(\sigma_j)$. 
	
	
	By setting $\bfalpha:=a((X,M),L)\bfbeta$, we find that $$\sum_{j=1}^k \alpha_j \pr^*_M L(\sigma_j)+D_{(X,M)}= a((X,M),L)\pr^*_M D+D_{(X,M)}$$ is effective, $\sum_{j=1}^k \alpha_j=a((X,M),L)$ and $\alpha_j>0$ for all $j=1,\dots, k$.
\end{proof}

\subsubsection{Choice of linear forms}
In this section we will choose the set $\mathscr{L}$ of linear forms, and we verify that conditions (P2) and (P3) in Theorem \ref{theorem: Breteche 1} are satisfied with this choice.
\begin{assumption}\label{assumption: divisor in interior}
	Without loss of generality, we assume that we have chosen the representative $D=a_1D_1+\dots+a_nD_n$ of $L$ such that the $\mathbb{Q}$-divisor
	$a((X,M),L)\pr^*_M D+D_{(X,M)}=\sum_{\bfm\in \Gamma_M} \tilde{a}_{\bfm} \tilde{D}_{\bfm}$ is effective and has maximal support: we assume $\tilde{a}_{\bfm}>0$ for as many $\bfm\in \Gamma_M$ as possible for such a representative of $L$.
\end{assumption} 
Let $(X,M^\circ)\subset (X,M)$ be the toric pair as in Definition \ref{def: alpha-constant} and let $\bfalpha$ be as in Corollary \ref{corollary: optimal choice alpha 2}.
For $\bfm\in \Gamma_M$, the coefficient of $\tilde{D}_\bfm$ in $a((X,M),L)\pr^*_M D+D_{(X,M)}$ is given by $l_{\bfm}(\bfalpha)-1$. Thus, by Assumption \ref{assumption: divisor in interior}, $\bfm\in \Gamma_M$ lies in $\fM^\circ$ if and only if $l_{\bfm}(\bfalpha)=1$. Let $$\mathscr{L}=\{l_{\bfm}\mid \bfm\in \Gamma_{M^\circ}\} \text{ and }\mathscr{R}=\{l^{(1)},\dots, l^{(n)}\}.$$

Since $a_i>0$ for all $i=1,\dots,n$, $l^{(i)}(\bfalpha)=\frac{a_i}{a((X,M),L)}>0$. Furthermore, if $\bfm,\bfm'\in \mathbb{N}^n$ satisfy $\bfm< \bfm'$ using the natural partial order on $\mathbb{N}^n$, then $\Re(l_{\bfm}(\bfs))< \Re(l_{\bfm'}(\bfs))$ for all $\bfs\in \mathbb{C}^k$ satisfying $\Re(l^{(i)}(\bfs))>0$. As $\fM_{\red,\mon}\setminus \fM^\circ$ has a finite number of minimal elements in this ordering on $\mathbb{N}^n$, the continuity of the linear forms $l_{\bfm}$ and $l^{(i)}$ implies that there exist $\frac{1}{4}>\delta_1>0, \epsilon>0$ such that for all
$$\bfs\in \mathscr{D}(\delta_1)=\{\bfs\in \mathbb{C}^k\mid \Re(l(\bfs))>-\delta_1 \, \forall l\in \mathscr{L}\cup \mathscr{R}\},$$
we have that $\Re(l_{\bfm}(\bfalpha+\bfs))>1+\epsilon$ for all $\bfm\in \fM_{\red, \mon}\setminus \fM^\circ$ and $\Re(l^{(i)}(\bfalpha+\bfs))>\epsilon$.

Consider the function $$H(\bfs)=F(\bfs+\bfalpha)\prod_{\bfm\in \Gamma_{M^\circ}} l_{\bfm}(\bfs),$$
and write
$$\frac{F(\bfs)}{\prod_{\bfm\in \Gamma_{M^\circ}} \zeta(l_{\bfm}(\bfs))}=\prod_{p\, \mathrm{prime}} G_p(\bfs),$$
where $$G_p(\bfs)=\left(\sum_{\bfm\in \fM_{\red}}p^{-l_{\bfm}(\bfs)}\right)\prod_{\bfm\in \Gamma_{M^\circ}}(1-p^{-l_{\bfm}(\bfs)})$$
for all prime numbers $p$ not dividing $S$.

The product $$\left(\sum_{\bfm\in \fM^\circ} p^{-l_{\bfm}(\bfs)}\right)\prod_{\bfm\in \Gamma_{M^\circ}}(1-p^{-l_{\bfm}(\bfs)})$$ is a finite sum of the form $1+\sum_{\bfm\in I} c_{\bfm} p^{-l_{\bfm}(\bfs)}$, where $I\subset \fM_{\red,\mon}$ is a finite set disjoint from $\Gamma_{M^\circ}$. In particular we see that the absolute value of
$\left(\sum_{\bfm\in \fM^\circ} p^{-l_{\bfm}(\bfalpha+\bfs)}\right)\prod_{\bfm\in \Gamma_{M^\circ}}(1-p^{-l_{\bfm}(\bfalpha+\bfs)})-1$ is bounded by $\# I \cdot p^{-1-\epsilon}$
for all $\bfs\in \mathscr{D}(\delta_1)$. By writing $$G_p(\bfs)=\left(\sum_{\bfm\in \fM^\circ}p^{-l_{\bfm}(\bfs)}+\sum_{\bfm\in \fM_{\red}\setminus \fM^\circ}p^{-l_{\bfm}(\bfs)}\right)\prod_{\bfm\in \Gamma_{M^\circ}}(1-p^{-l_{\bfm}(\bfs)}),$$
we find that the function $G_p(\bfs)$ satisfies $G_p(\bfs+\bfalpha)=1+O(p^{-1-\epsilon})$ for all $\bfs\in \mathscr{D}(\delta_1)$ and for all prime numbers $p$ not dividing $S$, where the implied constant depends on $\delta_1$ but is independent of the prime $p$. Thus the product $G(\bfs)=\prod_{p\, \mathrm{prime}}G_p(\bfs)$ converges to a bounded holomorphic function on $\bfalpha+\mathscr{D}(\delta_1)$. Since $s\zeta(s+1)$ is a holomorphic function, this implies that $H(\bfs)$ can be analytically continued to a function on $\mathscr{D}(\delta_1)$ and therefore condition (P2) of Theorem \ref{theorem: Breteche 1} is satisfied.

Now, as in de la Bretèche's work \cite[\S 4.3]{Bre01Toric}, we will use the upper bound
$$z\zeta(z+1)\ll (\Im z+1)^{1-\min(\Re z,0)/3+\epsilon}, \quad T \geq \Re(z)\geq -\tfrac{1}{2},$$
valid for all $\epsilon>0$ and $T>0$,
which follows from \cite[Theorem II.3.8]{Ten15}. By shrinking $\delta_1$ if necessary, we can assume that $G$ extends to a holomorphic function on the topological closure of $\bfalpha+\mathscr{D}(\delta_1)$.
Now since $G(\bfs)$ is bounded on $\bfalpha+\mathscr{D}(\delta_1)$, this implies that condition (P3) in Theorem \ref{theorem: Breteche 1} is satisfied with $\delta_2=1/3$.

Thus all conditions of Theorem \ref{theorem: Breteche 1} are satisfied with the choices made above for $\bfalpha, \delta_1, \delta_2, \mathscr{L}$ and $\mathscr{R}$.

\subsubsection{Determining the rank}
To finish the proof of Lemma \ref{lemma: weak form asymptotics}, we need to determine the rank of $\mathscr{L}$. We will first compute the rank of the matrix given by the linear forms $l^{(1)},\dots, l^{(n)}$. For vectors $l_1,\dots,l_n\in \mathbb{R}^k$, we write $(l_1,\dots,l_n)$ for the matrix with rows $l_1,\dots, l_n$.
\begin{proposition} \label{prop: rank matrix}
	If $L$ is a big and nef $\mathbb{Q}$-divisor class, then the rank of the matrix $(l^{(1)},\dots, l^{(n)})$ is $\dim(X)+1$.
	Consequently, any $\mathbb{R}$-divisor in $\Div(X)_{\mathbb{R}}$ that is $\mathbb{R}$-linearly equivalent to $0$ lies in $$V=\left\langle \sum_{j=1}^k y_j \pr^*_M L(\sigma_j)\,\middle\vert\, (y_1,\dots, y_k)\in \mathbb{Q}^k,\, \sum_{j=1}^k y_j=0 \right\rangle.$$
\end{proposition}
\begin{proof}
	Without loss of generality we can assume that $L$ is a divisor class, rather than just a $\mathbb{Q}$-divisor class. We represent $L$ by the divisor $D':=L(\sigma_1)=a_1'D_1+\dots+a_n'D_n$. Then we have $\mu_{D'}(\sigma_1)=0$, so the first column of the matrix $(l^{(1)},\dots, l^{(n)})$ is just $(a_1',\dots,a_n')$. Since $L$ is a nonzero divisor class, the divisor $L(\sigma_1)$ does not lie in the linear span of the divisors $L(\sigma_1)-L(\sigma_2),\dots, L(\sigma_1)-L(\sigma_k)$. Therefore $$\rank(l^{(1)},\dots, l^{(n)})=\rank(A)+1,$$ where $A$ is the matrix such that the coefficient in position $(i,j)$ is $\langle \mu_{D'}(\sigma_j), n_{\rho_i}\rangle$. Since a nef divisor on a toric variety is globally generated by \cite[Theorem 6.3.12]{CLS11}, we see by \cite[Theorem 6.1.7]{CLS11} that $\rank(A)=\dim P_L$, where $P_L$ is the polyhedron associated to $L$ as defined in \cite[\S 4.3]{CLS11}. By \cite[Lemma 9.3.9]{CLS11} we also have $\dim P_L=\dim (X)$, as $L$ is big.
	The vector space $V$ is contained in the kernel of $\Div_T(X)_\mathbb{Q}\rightarrow \Pic(X)_\mathbb{Q}$. By \cite[Theorem 4.2.1.]{CLS11} this kernel has dimension $\dim(X)$, but $\rank(l^{(1)},\dots, l^{(n)})=\dim(X)+1$ implies $V$ has dimension $\dim(X)$ as well, so $V$ is equal to the kernel.
\end{proof}
We view $\mathscr{L}$ as a linear map $\mathbb{Q}^k\rightarrow \mathbb{Q}^{\Gamma_{M^\circ}}$, so that $\#\Gamma_{M^\circ}-\rank(\mathscr{L})=\rank \coker \mathscr{L}$.
\begin{proposition} \label{prop: rank cokernel}
	The rank of $\coker \mathscr{L}$ is equal to $b(\mathbb{Q},(X,M),L)-1$. 
\end{proposition}
\begin{proof}
	The cokernel of $\mathscr{L}$ is the dual space of the kernel of the dual map $\mathscr{L}^\vee\colon \mathbb{Q}^{\Gamma_{M^\circ}}\rightarrow \mathbb{Q}^k$. This kernel is 
	\begin{equation*}
		\left\{\bfx\in\mathbb{Q}^{\Gamma_{M^\circ}}\middle\vert
		\sum_{\bfm\in \Gamma_{M^\circ}}  l_{\bfm}(\mathbf{e}_j)x_{\bfm}=0 \text{ for all }j\in \{1,\dots,k\}
		\right\}.
	\end{equation*}
	Recall that, for all $j\in \{1,\dots,k\}$, $L(\sigma_j)=\sum_{i=1}^n l^{(i)}(\bfe_j)D_i$ by the defining formula \eqref{Equation: l^{(i)}} for $l^{(i)}(\bfe_j)$, so $\pr^*_{M^{\circ}} L(\sigma_j)= \sum_{\bfm\in \Gamma_{M^{\circ}}} l_{\bfm} (\bfe_j)\tilde{D}_{\bfm}$. This implies that the kernel is isomorphic to
	$$\{\bfx\in \Div(X,M^{\circ})^\vee_{\mathbb{Q}}\mid \bfx(\pr^*_{M^\circ} L(\sigma_j))=0 \, \forall j\in \{1,\dots, k\}\}.$$
	Every such function is zero on torus-invariant principal divisors, as Proposition \ref{prop: rank matrix} implies that these are the divisors of the form $\sum_{j=1}^k y_j \pr^*_{M}L(\sigma_j)$ for $(y_1,\dots,y_k)\in \mathbb{Q}^k$ satisfying $\sum_{j=1}^k y_j=0$. This implies that the kernel is naturally identified with
	$$\{\bfx\in \Pic(X,M^\circ)^\vee_{\mathbb{Q}}\mid \bfx(\pr^*_{M^\circ} L)=0\}.$$
	As $(X,M^\circ)$ is quasi-proper with respect to $L$ as $(X,M)$ is quasi-proper, $\pr^*_{M^\circ} L$ is not $\mathbb{Q}$-linearly equivalent to zero, so $\rank \coker \mathscr{L}=\rank \Pic(X,M^\circ)-1$.
	
	The class of a torus-invariant prime divisor $\tilde{D}_{\bfm}\in \Div(X,M)$ is contained in the minimal face of $\Eff^1(X,M)$ containing $a((X,M),L)\pr^*_M L+K_{(X,M)}$ if and only if $\bfm\in \Gamma_M\setminus \Gamma_{M^\circ}$, by construction of the pair $(X,M^\circ)$. Since the effective cone $\Eff^1(X,M)$ is generated by torus-invariant divisors by Proposition \ref{prop: effective cone generated by torus-invariant}, this implies that $b(\mathbb{Q},(X,M),L)=\rank \Pic(X,M^\circ)$ finishing the proof.
\end{proof}

\begin{proof}[Proof of Lemma \ref{lemma: weak form asymptotics}]
	Theorem \ref{theorem: Breteche 1} implies that
	$$N_{(X,M),L,S}(B)=B^{a((X,M),L)}(Q(\log B)+O(B^{-\theta})),$$
	where $Q$ is a polynomial has degree at most $b(\mathbb{Q},(X,M),L)-1$ and $\theta>0$. This finishes the proof of the lemma.		
\end{proof}

\subsection{Proof of Theorem \ref{theorem: full asymptotics}, part II: the leading constant} \label{section: leading constant}
Now we will show that the polynomial $Q$ has the expected degree for any big and nef $\mathbb{Q}$-divisor $L$. We will furthermore compute the leading constant under the assumption that the $\mathbb{Q}$-divisor $L$ is toric adjoint rigid with respect to $(X,M)$.

We first notice that it suffices to prove Theorem \ref{theorem: full asymptotics} for adjoint rigid and toric adjoint rigid $\mathbb{Q}$-divisors $L$ satisfying $a((X,M),L)=1$.
\begin{proposition}
	To show that the polynomial $Q$ obtained in Lemma \ref{lemma: weak form asymptotics} has degree $b(\mathbb{Q},(X,M),L)-1$, it suffices to assume that $a((X,M),L)=1$, $S=1$ and $L$ is toric adjoint rigid with respect to $(X,M)$.
\end{proposition}
\begin{proof}
	We can assume without loss of generality that $a((X,M),L)=1$ since the height function satisfies $H_{tL}(\bfx)=H_{L}(\bfx)^t$ for any $t\in \mathbb{Q}_{>0}$, and thus $N_{(X,M),tL,S}(B)=N_{(X,M),L,S}(B^t)$.
	The pair $(X,M^\circ)$ is a pair that is quasi-proper with respect to $L$, and $L$ is toric adjoint rigid with respect to $(X,M^\circ)$. Furthermore $a((X,M),L)=a((X,M^\circ),L)$ and $b(\mathbb{Q},(X,M),L)=b(\mathbb{Q},(X,M^\circ),L)=\rank \Pic(X,M^\circ)$ as in the proof of Proposition \ref{prop: rank cokernel}. By Lemma \ref{lemma: weak form asymptotics}, we know that $N_{M,L,S}(B)=B^{a((X,M),L)} (Q(\log B)+O(B^{-\theta}))$ as $B\rightarrow \infty$ for some $\theta>0$ and some polynomial $Q$ of degree at most $b(\mathbb{Q},(X,M),L)-1$.
	Note that furthermore $$N_{(X,M),L,S}(B)\geq N_{(X,M^\circ),L,1}(B).$$
	
	If we assume that Theorem \ref{theorem: full asymptotics} is true if $a((X,M),L)=1$, $S=1$ and $L$ is toric adjoint rigid with respect to $(X,M)$, then this implies $$N_{(X,M^\circ),L,1}(B)=B^{a((X,M),L)}(Q'(\log B)+O(B^{-\theta})),$$
	for some polynomial $Q'$ of degree $b(\mathbb{Q},(X,M),L)-1$. Now the basic inequality
	$$N_{(X,M),L,S}(B)\geq N_{(X,M^\circ),L,1}(B)$$
	implies that the degree of $Q$ is at least the degree of $Q'$, and thus $Q$ has degree $b((X,M),L)-1$ as well.
\end{proof}

\begin{assumption}
	Henceforth we assume that $L$ is adjoint rigid with respect to $(X,M)$, or that it is toric adjoint rigid with respect to $(X,M)$ and $S=1$, and in these cases we will compute the leading coefficient of the polynomial $Q$.
\end{assumption}
\begin{assumption}
	The constant $C_\infty$ in Theorem \ref{theorem: full asymptotics} does not depend on the choice of the integers $d_1,\dots, d_n$ determining the pair $(X,\overline{M^\circ})$ as in Notation \ref{notation: fan toric pair}. Therefore, we will assume that the integers are chosen large enough to ensure $a(\mathbb{Q},(X,M),L)=a(\mathbb{Q},(X,\overline{M^\circ}),L)$. Note that such integers exist as $(X,M^\circ)$ is quasi-proper with respect to $L$.
\end{assumption}
We will prove Theorem \ref{theorem: full asymptotics} using another theorem of de la Bretèche.
\begin{theorem}\cite[Théorème 2(ii), Remarques (ii)]{Bre01Sum} \label{theorem: Breteche 2}
	In the setting of Theorem \ref{theorem: Breteche 1}, assume that the following additional conditions are satisfied:
	\begin{enumerate}
		\item [(C1)] There exists a function $\tilde{H}$ such that $H(\bfs)=\tilde{H}(l_1(\bfs),\dots,l_{\tilde{n}}(\bfs))$;
		\item [(C2)] the vector $(1,\dots,1)\in \mathbb{R}^k$ is a strictly positive linear combination of $l_1,\dots, l_{\tilde{n}}$;
		\item [(C3)] $l_1(\bfalpha)=\dots=l_{\tilde{n}}(\bfalpha)=1$.
	\end{enumerate}
	Then the polynomial $Q$ satisfies the relation
	$$Q(\log B)=C_0 B^{-\sum_{j=1}^k \alpha_j}\mathrm{Volume}(D(B))+O(\log(B)^{\rho-1}),$$
	as $B\rightarrow \infty$. Here $\rho:=\tilde{n}-\rank(l_1,\dots,l_{\tilde{n}})$, $C_0:=H(0,\dots,0)$ and
	$$D(B)=\left\{\mathbf{y}\in [1,\infty)^{\tilde{n}}\middle|\; \prod_{i=1}^{\tilde{n}} y_i^{l_i(\mathbf{e}_j)}\leq B\quad \forall j=1,\dots,k\right\}.$$
\end{theorem}
We will apply this theorem for the same series as in the proof of Lemma \ref{lemma: weak form asymptotics}. We thus only need to verify conditions (C1), (C2) and (C3) and then estimate the volume and show that $H(0,\dots,0)\neq 0$. Due to the way we chose $\mathscr{L}$, condition (C3) is trivially satisfied. We first show that condition (C2) is satisfied.
\begin{proposition}
	If $(X,M)$ is a smooth toric pair such that $0\in \Pic(X,M)$ is a toric rigid divisor, then the monoid $N_{M}^+\subset N$ introduced in Definition \ref{def: invariants} is a lattice.
\end{proposition}
\begin{proof}
	We argue by contradiction, and assume that $N_{M}^+$ is not a lattice. Then the cone generated by $N_{M}^+$ is not a vector space, and thus there exists a linear form $f\colon N_{\mathbb{R}}\rightarrow \mathbb{R}$ such that the half-space $H=\{\mathbf{n}\in N_{\mathbb{R}}\mid f(\mathbf{n})\geq 0\}$ contains $N_{M}^+$ but not $-N_{M}^+$. Since $(X,M)$ is smooth, $N_M^+$ is finitely generated, so the linear form $f$ can be chosen such that it restricts to a homomorphism $N\rightarrow \mathbb{Z}$, i.e. to an element in $N^\vee$. Using the description of $\Pic(X,M)$ given in Proposition \ref{prop: Picard group toric pair}, this implies that the divisor $\sum_{\bfm\in \Gamma_{M}} f(\phi(\bfm)) \tilde{D}_\bfm$ is linearly equivalent to $0$. By construction $f(\phi(\bfm))\geq 0$ for all $\bfm\in \Gamma_M$ and $f(\phi(\bfm))\neq 0$ for some $\bfm\in \Gamma_M$, so this is a nontrivial torus-invariant effective divisor. This is in contradiction with the fact that $0$ is toric adjoint rigid with respect to $(X,M)$.
\end{proof}
Using the above proposition and the fact that $L$ is toric adjoint rigid with respect to $(X,M)$, there exist coefficients $c_{\bfm}>0$ corresponding to the generators $\bfm\in \Gamma_{M^\circ}$ such that $\sum_{\bfm\in \Gamma_{M^\circ}} c_{\bfm} \phi(\bfm)=0$, where $\phi\colon \mathbb{N}^n\rightarrow N$ is the homomorphism in Definition \ref{def: invariants}. Therefore the sum $$\sum_{\bfm\in \Gamma_{M^\circ}} l_{\bfm}(\mathbf{e}_j)c_\bfm=\sum_{\bfm\in \Gamma_{M^\circ}} (a_\bfm-\langle \mu_L(\sigma_j),\phi(\bfm)\rangle)c_\bfm=\sum_{\bfm\in \Gamma_{M^\circ}}c_\bfm \sum_{i=1}^n a_im_i>0$$ does not depend on $j$, 
and thus
$$\sum_{\bfm\in \Gamma_{M^\circ}} l_{\bfm} c_\bfm=\sum_{\bfm\in \Gamma_{M^\circ}}c_\bfm \sum_{i=1}^n a_im_i \cdot(1,\dots,1),$$
so
$\mathcal{B}=(1,\dots,1)$ is a positive linear combination of the linear forms in $\mathscr{L}$, and hence condition (C2) of Theorem \ref{theorem: Breteche 2} is satisfied.
Finally, condition (C1) will follow from the following proposition.
\begin{proposition} \label{prop: characterisation adjoint rigid}
	The $\mathbb{Q}$-divisor $L$ is adjoint rigid with respect to $(X,M)$ if and only if for every $i=1,\dots,n$ the linear form $l^{(i)}$ lies in the linear span of $\mathscr{L}$.
	Similarly, the divisor $L$ is toric adjoint rigid with respect to $(X,M)$ if and only if for every $\bfm\in \fM$ the linear form $l_{\bfm}$ lies in the linear span of $\mathscr{L}$.
\end{proposition}
\begin{proof}
	We give the proof for the toric adjoint rigid case, and we note that the adjoint rigid case is proved analogously. By Corollary \ref{corollary: optimal choice alpha 2}, any representative $D=a_1D_1+\dots+a_nD_n$ of $a((X,M),L)L$ with $a_1,\dots,a_n>0$ can be written as $D=\sum_{j=1}^k \alpha_j L(\sigma_j)$ with $\alpha_1,\dots,\alpha_k>0$, for some solution $\bfalpha$ of the linear program $\mathscr{P}_M$.
	From this we obtain the expression $$\pr^*_M D+D_{(X,M)}=\sum_{\bfm\in \Gamma_M} (l_{\bfm}(\bfalpha)-1) \tilde{D}_{\bfm},$$ 
	which implies that $L$ is toric adjoint rigid with respect to $(X,M)$ if and only if for every $\bfm\in \fM$ the value $l_{\bfm}(\bfalpha)$ does not depend on the choice of a solution $\bfalpha$ to the linear program $\mathscr{P}_M$.
	
	If the linear form $l_{\bfm}$ lies in the linear span of $\mathscr{L}$ for every $\bfm\in \fM$, then the the values of the linear forms in $\mathscr{L}$ evaluated at $\bfalpha$ determine the value of $l_{\bfm}(\bfalpha)$ for all $\bfm\in \fM$.
	Since for every $l\in \mathscr{L}$ we have $l(\bfalpha)=1$ by definition, we therefore see that $L$ is toric adjoint rigid.
	
	Conversely, assume that $l_{\bfm'}$ is a linear form not in the span of $\mathscr{L}$ for some $\bfm'\in \fM$. Then there exists $\bfbeta\in \mathbb{R}^k$ such that $l_{\bfm'}(\bfbeta)=1$ but $l(\bfbeta)=0$ for all $l\in \mathscr{L}$. Let $D$ be a representative of $L$ satisfying Assumption \ref{assumption: divisor in interior}, and take $\bfalpha$ such that $D=\sum_{j=1}^k \alpha_j L(\sigma_j)$ and $\alpha_j>0$ for all $j=1,\dots,k$. For every $\bfm\in \Gamma_M$ such that $l_{\bfm}\not\in \mathscr{L}$ we have $l_{\bfm}(\bfalpha)>1$, by construction of $\mathscr{L}$. Therefore, there exists $\epsilon>0$ such that $l_{\bfm}(\bfalpha+\epsilon\bfbeta)\geq 1$ for all $\bfm\in \fM$ and $\bfalpha+\epsilon\bfbeta>0$. This implies that $\bfalpha+\epsilon\bfbeta\in \mathbb{R}_{>0}^{k}$ is also a solution to $\mathscr{P}_M$. Since $l_{\bfm'}(\bfalpha+\epsilon\bfbeta)\neq l_{\bfm'}(\bfalpha)$, this implies that $L$ is not toric adjoint rigid.
\end{proof}

Note that $F(\bfs)$ is always a function of the linear forms $l^{(1)}(\bfs),\dots, l^{(n)}(\bfs)$. Furthermore if $S=1$, then $F(\bfs)$ is a function of the linear forms $l_{\bfm}(\bfs)$ for $\bfm\in \Gamma_M$. Hence condition (C1) is satisfied by Proposition \ref{prop: characterisation adjoint rigid}, since either $L$ is adjoint rigid or $L$ is toric adjoint rigid and $S=1$. Thus we can apply Theorem \ref{theorem: Breteche 2} to determine the leading constant. By this theorem, the polynomial $Q$ giving the asymptotic satisfies $$Q(\log B)=2^{\dim X}C_0I(B)/B^{a((X,M),L)}+O((\log B)^{b(\mathbb{Q},(X,M),L)-2}),$$
where $$C_0=H(0,\dots,0)$$
and $I(B)$ is the volume of the domain
$$D(B)=\left\{\mathbf{x}\in [1,\infty)^{\Gamma_{M^\circ}}\middle| \prod_{\bfm\in \Gamma_{M^\circ}} x_\bfm^{l_{\bfm}(\mathbf{e}_j)}\leq B,\, \forall j=1,\dots, k\right\}.$$
Thus to prove Theorem \ref{theorem: full asymptotics}, it remains to compute $C_0$ and estimate $I(B)$ as $B\rightarrow\infty$. First we will compute $C_0$.
\begin{proposition} \label{prop: C0}
	The value of $H$ at the origin is equal to the infinite product
	\begin{align*}
		C_0=&\prod_{\substack{p \, \mathrm{prime}\\ p\mid S}}(1-p^{-1})^{\#\Gamma_{M^\circ}} \sum_{\bfm\in \fM_{\red}}p^{-a_{\bfm}}\\
		\times &\prod_{\substack{p \, \mathrm{prime}\\ p\nmid S}}(1-p^{-1})^{\#\Gamma_{M^\circ}} \sum_{\bfm\in \mathbb{N}^n_{\red}}p^{-a_{\bfm}}
	\end{align*}
	and this quantity is positive. Here we recall $a_{\bfm}=\sum_{i=1}^n m_i a_i$. Furthermore, if $L$ is adjoint rigid, then $\#\Gamma_{M^\circ}=\dim(X)+b(\mathbb{Q},(X,M),L)$.
\end{proposition}
\begin{proof}
	Note that
	$$H(\bfs)=G(\bfs+\bfalpha)\prod_{\bfm\in\Gamma_{M^\circ}}l_{\bfm}(\bfs)\zeta(l_{\bfm}(\bfs)+1),$$
	and $$G_p(\bfs)=\left(\sum_{\bfm\in \fM}p^{-l_{\bfm}(\bfs)}\right)\prod_{\bfm\in \Gamma_{M^\circ}}(1-p^{-l_{\bfm}(\bfs)})$$
	for any prime number $p$ not dividing $S$ and similarly
	$$G_p(\bfs)=\left(\sum_{\bfm\in \mathbb{N}^n_{\red}}p^{-l_{\bfm}(\bfs)}\right)\prod_{\bfm\in \Gamma_{M^\circ}}(1-p^{-l_{\bfm}(\bfs)})$$
	for any prime number $p$ dividing $S$.
	
	Thus the limit $\lim_{z\rightarrow 0} z\zeta(z+1)=1$ implies $H(0,\dots,0)=G(\bfalpha)$. This gives the desired identity for $C_0$ under the assumption that the product converges.
	Each term $C_p$ in the product is positive and $C_p=1+O(p^{-c})$, where the implied constant is independent of $p$ and $c$ is the smallest between $2$ and the minimum of all values $l_{\bfm}(\bfalpha)= a_{\bfm}$ for $\bfm\in \fM\setminus \fM^\circ$, so the product thus converges to a positive constant.
	
	Now assume that $L$ is adjoint rigid with respect to $(X,M)$. We claim that this implies that the map $N^\vee\rightarrow \Div_T(X,M)$ is injective. If it were not injective, then there exists $\mu\in N^\vee$ such that $\pr^*_M\div(\chi^\mu)=0$. But then there exists $c\in \mathbb{Q}$ such that $\pr^*_M\div(c+\chi^\mu)$ is a nontrivial effective divisor on $(X,M)$. This implies that $0\in \Div(X,M)$ is not a rigid divisor, which contradicts the fact that $L$ is adjoint rigid with respect to $(X,M)$. Thus Proposition \ref{prop: Picard group toric pair} implies $\#\Gamma_{M^\circ}=\dim(X)+\rank \Pic(X,M^\circ)$. Furthermore, the proof of Proposition \ref{prop: rank cokernel} shows $\rank \Pic(X,M^\circ)=b(\mathbb{Q},(X,M),L)$, so we obtain the desired expression for $\#\Gamma_{M^\circ}$. 
\end{proof}

Now it remains to estimate the volume of the set $D(B)$.
In order to simplify notation, we assume $M=M^\circ$, which we can do without loss of generality as the definition of $D(B)$ only depends on $(X,M^\circ)$ and not on $(X,M)$.
The set $D(B)$ is a generalization of the set $D(B)$ defined by Salberger \cite[Notation 11.28]{Sal98} in his study of rational points on split toric varieties, and we will use the same approach he used to estimate its volume.

We regard $D(B)$ as a closed subset of the real locus of the universal torsor of $(X,M)$, where the universal torsor is as in Definition \ref{def: universal torsor pair}.
As we assumed $M=M^\circ$ and $a((X,M),L)=1$, we have $\pr^*_{M} D=-D_{(X,M)}$.
In order to estimate the volume of $D(B)$, Salberger splits it up as $D(B)=\cup_{\sigma\in \Sigma_{\max}} D(B,\sigma)$ using what he calls the toric canonical splitting \cite[Notation 11.31]{Sal98}. 
We will similarly split up $D(B)$, but we will use the splitting induced by the fan $\Sigma_{\overline{M}}$ as given in Notation \ref{notation: fan toric pair}, rather than $\Sigma$. The fan $\Sigma_{\overline{M}}$ has the property that for all $\bfm\in \Gamma_M$, the ray spanned by an element $\phi(\bfm)\in N$ lies in $\Sigma_{\overline{M}}$, and all rays in $\Sigma_{\overline{M}}$ are of this form. This property will aid in computing $D(B)$.
Since the dense torus in $X$ is $U=\Hom(N^\vee,\mathbb{G}_m)$, the real locus of the torus is $U(\mathbb{R})=\Hom(N^\vee,\mathbb{R}^\times)$. By composing with the logarithm of the absolute value $\mathbb{R}^\times\rightarrow \mathbb{R}\colon x\mapsto \log |x|$, we obtain a homomorphism $$U(\mathbb{R})\rightarrow \Hom(N^\vee,\mathbb{R}^\times)= N_{\mathbb{R}}.$$
Recall that $U_M$ is the dense torus in the universal torsor $Y_M$ of $(X,M)$. The morphism $Y_M\rightarrow X$ induces a homomorphism $U_M(\mathbb{R})\rightarrow U(\mathbb{R})$. By composing these homomorphism with each other, we obtain a homomorphism
$$\operatorname{Log}_M\colon U_M(\mathbb{R})\rightarrow \Hom(N^\vee,\mathbb{R}^\times)= N_{\mathbb{R}}.$$
For $\sigma\in \Sigma_{\overline{M},\max}$, write $C_{M,\sigma,0}(\mathbb{R})$ for the inverse image of $-\sigma$ under the map $\operatorname{Log}_M$ and write $D(B,\sigma)=D(B)\cap C_{M,\sigma,0}(\mathbb{R})$. Since for any two distinct maximal cones $\sigma,\sigma'$, their intersection $\sigma\cap \sigma'$ lies in a proper subspace of $N_{\mathbb{R}}$, the intersection of $D(B,\sigma)\cap D(B,\sigma')$ has Lebesgue measure zero and thus
\begin{equation} \label{eq: subdivide integral D(B) over cones}
	I(B)=\int_{D(B)}\d \bfx=\sum_{\sigma\in \Sigma_{\overline{M},\max}}\int_{D(B,\sigma)}\d \bfx,
\end{equation}
where the measure is the standard Lebesgue measure on $\mathbb{R}^{\Gamma_M}$.
We will compute $\int_{D(B,\sigma)}\d \bfx$ for each maximal cone $\sigma\in \Sigma_{\overline{M}}$.

As in \cite[Proposition 11.22]{Sal98}, we can describe when a point lies in $C_{M,\sigma,0}(\mathbb{R})$. The cone $\sigma$ contains exactly $d=\dim X$ rays. Let $d_1$ be the number of rays in $\sigma$ which lie in $\Sigma_M$. Let $r=\#\Sigma_{M(\sigma)}(1)-d$ be the number of raysin $\Sigma_{M(\sigma)}$ which lie outside of $\sigma$. Order the rays $\rho_1,\dots, \rho_{r+d}$ in $\Sigma_{M(\sigma)}$ such that $\rho_{r+1},\dots,\rho_{r+d_1}$ are the rays in both $\sigma$ and $\Sigma_M$ and $\rho_{r+1},\dots,\rho_{r+d}$ are the rays in $\sigma$. Let $\bfm_{1},\dots,\bfm_{r+d}\in \Gamma_{M(\sigma)}$ be the elements corresponding to the rays $\rho_{1},\dots \rho_{r+d}$, and set $n^{(i)}=\phi(\bfm_{r+i})$ for $i=1,\dots, d$. As $n^{(1)},\dots, n^{(d)}$ are integer multiples of the ray generators of $\rho_{r+1},\dots,\rho_{r+d}$, they freely generate a finite-index sublattice $N_{\sigma}$ of $N$. Let $(\mu^{(1)},\dots, \mu^{(d)})$ be the corresponding dual $\mathbb{Z}$-basis of $N_{\sigma}^\vee\supset N^\vee$ and set

$$D(i)=\sum_{\bfm\in\Gamma_{M}} \langle \mu^{(i)}, \phi(\bfm) \rangle \tilde{D}_{\bfm}\in \Div_T(X,M)_{\mathbb{Q}},$$
where $\tilde{D}_\rho$ is the prime divisor on $(X,M)$ corresponding to the ray $\rho$.
Note that
$$D(i)=\pr^*_M \sum_{\rho\in \Sigma(1)} \langle \mu^{(i)}, n_{\rho} \rangle D_{\rho}$$
by construction, so $D(i)$ is $\mathbb{Q}$-linearly equivalent to $0$. Furthermore, since $\mu^{(1)},\dots, \mu^{(d)}$ is a basis for $N^\vee_{\mathbb{Q}}$, $(D(1),\dots,D(n))$ is a basis for the vector space of all torus-invariant $\mathbb{Q}$-divisors on $(X,M)$ linearly equivalent to $0$.
\begin{proposition}\label{prop: subdivision D(i)}
	Let $\bfx\in [1,\infty)^{\Gamma_M}$. Then $\bfx\in C_{M,\sigma,0}(\mathbb{R})$ if and only if $\bfx^{D(i)}\leq 1$ for all $i=1,\dots, d$.
\end{proposition}
\begin{proof}
	The proof is identical to the proof of \cite[Proposition 11.22]{Sal98}.
\end{proof}

For $i\in \{1,\dots, r+d_1\}$, we write $\tilde{D}_i:=\tilde{D}_{\bfm_i}\in \Div(X,M)$ for the divisor corresponding to the ray $\rho_i\subset \Sigma_M$.

Set $$E(i):=\tilde{D}_{r+i}-D(i) \text{ if } \rho_{r+i}\in \Sigma_{M}(1)$$ and
$$E(i):=-D(i) \text{ if } \rho_{r+i}\in \Sigma_{\overline{M}}(1)\setminus \Sigma_{M}(1).$$
By construction, the divisor $E(i)$ is supported on the divisors $\tilde{D}_{\bfm}\in \Div(X,M)$ such that $\phi(\bfm)\not\in\sigma$ and it is $\mathbb{Q}$-linearly equivalent to $\tilde{D}_{r+i}$ if $\rho_{r+i}\in \Sigma_M(1)$ and otherwise it is $\mathbb{Q}$-linearly equivalent to $0$.


\begin{notation}
	Note that for a maximal cone $\sigma \in \Sigma_{\overline{M}}$ and a $\mathbb{Q}$-divisor class $L'\in \Pic(X,\overline{M})_{\mathbb{Q}}$, there is a unique representative $L'(\sigma)\in \Div_T(X,M)_{\mathbb{Q}}$ of $L'$ supported only on the divisors $\tilde{D}_{\bfm}\in \Div(X,\overline{M})$ with $\phi(\bfm)\not\in \sigma$.
	For a maximal cone $\sigma\in \Sigma_{\overline{M}}$, we write $(\pr^*_M L)(\sigma)$ for the restriction of $(\pr^*_{\overline{M}} L)(\sigma)$ to $(X,M)$. Similarly, we write $D_{(X,M)}(\sigma)\in \Div(X,M)_{\mathbb{Q}}$ by viewing $D'=D_{(X,M)}$ as a divisor on $(X,\overline{M})$ and by restricting $D'(\sigma)$ to $(X,M)$.
\end{notation}
In particular, for a maximal cone $\sigma\in \Sigma_{\overline{M}}$ and $L\in \Pic(X)_{\mathbb{Q}}$,
$$(\pr_M^* L)(\sigma)=\pr^*_M (L(\overline{\sigma})),$$
where $\overline{\sigma}$ is the unique maximal cone in $\Sigma$ containing $\sigma$.
\begin{lemma}
	$D(B,\sigma)$ is the set of all $(x_1,\dots,x_{r+d_1})\in X_{M,0}(\mathbb{R})\subset \mathbb{R}^{r+d_1}$ satisfying
	\begin{enumerate}
		\item $\min(x_1,\dots, x_{r+d_1})\geq 1$,
		\item $\bfx^{(\pr_M^*L)(\sigma)}\leq B$, 
		\item $\bfx^{E(i)}\geq x_{r+i}, \text{ for all }i=1,\dots,d_1$, and $\bfx^{E(i)}\geq 1, \text{ for all }i=d_1+1,\dots,d$.
	\end{enumerate}
\end{lemma}
\begin{proof}
	By Proposition \ref{prop: subdivision D(i)}, $\bfx\in [1,\infty)^{\Gamma_M}$ lies in $\bfx\in C_{M,\sigma,0}(\mathbb{R})$ if and only if the first and third conditions are satisfied. Since $\bfx^{(\pr^*_M L)(\sigma)}=\prod_{\bfm\in \Gamma_M} x_{\bfm}^{l_{\bfm}(\bfe_j)}$ for the unique maximal cone $\sigma_j\in \Sigma$ containing $\sigma$, it remains to show that $\bfx^{(\pr_M^*L)(\sigma)}\leq B$ is equivalent to $\bfx^{(\pr_M^*L)(\sigma')}\leq B$ for all $\sigma'\in \Sigma_{\overline{M}}$. Because the divisors $D(1),\dots, D(d)$ generate the kernel of $\Div_T(X,M)_{\mathbb{Q}}\rightarrow \Pic(X,M)_{\mathbb{Q}}$, we must have $(\pr_M^* L)(\sigma')=(\pr_M^*L)(\sigma)+\sum_{i=1}^d c_iD(i)$, for $c_1,\dots,c_d\in \mathbb{Q}$. Let $\overline{\sigma}$ and $\overline{\sigma'}$ be the unique maximal cones in $\Sigma$ containing $\sigma$ and $\sigma'$, respectively. By considering the pullbacks of $L(\overline{\sigma})$ and $L(\overline{\sigma'})$ to $(X,M)$, $c_i$ is equal to the coefficient of $\tilde{D}_{r+i}\in \Div(X,M)$ in the pullback of $L(\overline{\sigma'})-L(\overline{\sigma})$ to $(X,M)$, for all $i\in \{1,\dots,d\}$. As the coefficient of $\tilde{D}_{r+i}\in \Div(X,M)$ in $\pr^*_{M}L(\overline{\sigma})$ is zero, and $L$ is nef, me must have $c_i\geq 0$ for all $i\in\{1,\dots, d\}$. Therefore Proposition \ref{prop: subdivision D(i)} implies $$\bfx^{(\pr_M^* L)(\sigma')}\leq \bfx^{(\pr_M^* L)(\sigma)},$$
	as desired.
\end{proof}

Similarly, we define $\Omega(B,\sigma)$ to be the set of all $(x_1,\dots,x_r)$ satisfying
\begin{enumerate}
	\item $\min(x_1,\dots, x_{r})\geq 1$,
	\item $\bfx^{(\pr_M^*L)(\sigma)}\leq B$, 
	\item $\bfx^{E(i)}\geq 1, \text{ for all }i=1,\dots,d$.
\end{enumerate}

By Fubini's theorem, we can compute $D(B,\sigma)$ by first integrating with respect to $(x_{r+1},\dots,x_{r+d_1})$ and then with respect to $(x_1,\dots,x_r)$:
\begin{equation} \label{eq: D(B) and E(j)}
	\int_{D(B,\sigma)} \d \bfx= \int_{\Omega(B,\sigma)}\prod_{i=1}^{d_1} (\bfx^{E(i)}-1)\d x_1\dots \d x_r
\end{equation}
This equality combined with
\begin{equation} \label{eq: sum E(j)}
	\sum_{i=1}^{d_1} E(i)=-D_{(X,M)}(\sigma)-\sum_{i=1}^r\tilde{D}_i
\end{equation}
implies
\begin{equation} \label{eq: D(B) and E(j) 2}
	\int_{D(B,\sigma)} \d \bfx= \int_{\Omega(B,\sigma)} \bfx^{-D_{(X,M)}(\sigma)}\prod_{i=1}^{d_1} (1-\bfx^{-E(i)})\frac{\d x_1}{x_1}\dots \frac{\d x_r}{x_r}.
\end{equation}
Let $T_{M(\sigma)}\subset Y_{M(\sigma)}$ be the Picard torus for the pair $(X,M(\sigma))$ as in Definition \ref{def: universal torsor pair}. The projection $Y_{M(\sigma)}\subset \mathbb{A}_{\mathbb{Q}}^{r+d}\rightarrow \mathbb{A}_{\mathbb{Q}}^r$ onto the first $r$ coordinates induces an analytic homomorphism
$$T_{M(\sigma)}(\mathbb{R})\rightarrow (\mathbb{R}^\times)^r$$
$$(x_1,\dots,x_{r+d})\mapsto (x_1,\dots,x_r)$$
of Lie groups. The identity component $T_{M(\sigma)}(\mathbb{R})^+$ of $T_{M(\sigma)}(\mathbb{R})\subset \mathbb{R}^n$ is the set of points with positive coordinates, and the analytic homomorphism restricts to an analytic isomorphism $T_{M(\sigma)}(\mathbb{R})^+\rightarrow \mathbb{R}_{>0}^r$.
For $i\in \{1,\dots, d_1\}$, the image of $E(i)$ in $\Div_T(X,M(\sigma))$ is $\mathbb{Q}$-linearly equivalent to $\tilde{D}_{r+i}$, so $\bfx^{E(i)}=x_{r+i}$ for all $\bfx\in T_{M(\sigma)}(\mathbb{R})^+$. Furthermore, since $-D_{(X,M)}(\sigma)$ is $\mathbb{Q}$-linearly equivalent to $-D_{(X,M)}$, viewed as $\mathbb{Q}$-divisors on $(X,M(\sigma))$, the isomorphism identifies the set $\Omega(B,\sigma)$ with the subset $F(B)\subset T_{M(\sigma)}(\mathbb{R})$ given by the elements $(x_1,\dots,x_{r+d})$ with
\begin{enumerate}
	\item $\min(x_1,\dots, x_{r})\geq 1$,
	\item $\bfx^{(\pr_M^* L)(\sigma)}\leq B$.
\end{enumerate}

Under the isomorphism $T_{M(\sigma)}(\mathbb{R})^+\rightarrow \mathbb{R}_{>0}^r$, the differential form $\frac{\d x_1}{x_1}\dots \frac{\d x_r}{x_r}$ on $(\mathbb{R}^\times)^r$ corresponds to the torus-invariant differential form $\frac{\d x_1}{x_1}\dots \frac{\d x_r}{x_r}$ on $T_{M(\sigma)}(\mathbb{R})^+$.

Consequently, we find the following analogue of \cite[Equation (11.37)]{Sal98}
\begin{equation} \label{eq: full integral F(B)}
	\int_{D(B,\sigma)} \d \bfx=\int_{F(B)} \bfx^{-D_{(X,M)}}\prod_{i=1}^{d_1}\left(1-1/x_{r+i}\right)\frac{\d x_1}{x_1}\dots \frac{\d x_r}{x_r}.
\end{equation}

We will now first focus on estimating
\begin{equation} \label{eq: I(B,sigma) in terms of F(B)}
	I(B,\sigma)=\int_{F(B)} \bfx^{-D_{(X,M)}}\frac{\d x_1}{x_1}\dots \frac{\d x_r}{x_r},
\end{equation}
and then we show in Lemma \ref{lemma: reduction to I(B,sigma)} that $\int_{D(B,\sigma)} \d \bfx\sim I(B,\sigma)$ as $B\rightarrow \infty$.

There is an analytic isomorphism $$\psi\colon T_{M(\sigma)}(\mathbb{R})^+\rightarrow V:=\Hom(\Pic(X,M(\sigma)),\mathbb{R})$$
given by $(x_1,\dots,x_{r+d})\mapsto (y_1,\dots, y_{r+d})$, where $y_i=\log x_i$ for $i\in \{1,\dots, r+d\}$.

Let $\Hom_{\geq 0}(\Pic(X,M(\sigma)),\mathbb{R})\subset V$ be set of linear functions which are nonnegative on effective divisor classes on $(X,M(\sigma))$. Then the isomorphism $\psi$ sends the set $T_{M(\sigma)}^{\geq 1}(\mathbb{R})$ consisting of all $(x_1,\dots,x_{r+d})\in T_{M(\sigma)}(\mathbb{R})^+$ with $x_1,\dots, x_{r+d}\geq 1$ to $\Hom_{\geq 0}(\Pic(X,M(\sigma)),\mathbb{R})$.

Let $b=\log B$. The image $E_b:=\psi(F(B))$ is the set of all $\varphi\in \Hom_{\geq 0}(\Pic(X,M(\sigma)),\mathbb{R})$ with $\varphi(\pr_M^* L)\leq b$. Let $\nu$ be the Haar measure on $V$ such that the volume of $V/\Lambda$ is $1$ for the lattice $\Lambda:=\Hom(\Pic(X,M(\sigma)),\mathbb{Z})$.
Under the analytic isomorphism $\psi$, the differential form $\frac{\d x_1}{x_1}\dots \frac{\d x_r}{x_r}$ gets sent to $\d y_1 \dots \d y_r$. Recall that $I(\sigma)$ is the index of $\langle \tilde{D}_1,\dots,\tilde{D}_r \rangle$ inside $\Pic(X,M(\sigma))$. As $\langle \tilde{D}_1,\dots,\tilde{D}_r \rangle$ is torsion-free, it has index $\frac{I(\sigma)}{\#\Pic(X,M(\sigma))_{\mathrm{torsion}}}$ in $\Pic(X,M(\sigma))/\{\mathrm{torsion}\}$. Thus the lattice $\Lambda$ has index $\frac{I(\sigma)}{\#\Pic(X,M(\sigma))_{\mathrm{torsion}}}$ in $\langle [\tilde{D}_1]^*,\dots,[\tilde{D}_r]^* \rangle$.
This implies
$$\d y_1 \dots \d y_r=\frac{I(\sigma)}{\#\Pic(X,M(\sigma))_{\mathrm{torsion}}}\d \nu,$$ and thus Equation \eqref{eq: I(B,sigma) in terms of F(B)} becomes
\begin{equation}
	I(B,\sigma)=\frac{I(\sigma)}{\#\Pic(X,M(\sigma))_{\mathrm{torsion}}} \int_{E_b} \exp(y_1+\dots+y_{r+d_1})\d \nu.
\end{equation}

We write $\tilde{a}_i$ for the coefficient of $\tilde{D}_i$ in $\pr^*_{M(\sigma)} L$.
Note that $\pr^*_{M(\sigma)} L+D_{(X,M)}=\sum_{i=r+d_1+1}^{r+d} \tilde{a}_i\tilde{D}_i$, and $\tilde{a}_i=1$ if $i\leq r+d_1$. 

Let $V''\subset \mathbb{R}^{d-d_1}$ be the vector space of linear functions $\langle [\tilde{D}_{r+d_1+1}],\dots, [\tilde{D}_{r+d}]\rangle\rightarrow \mathbb{R}$.
The projection $V\rightarrow V''$ given by $(y_1,\dots, y_{r+d})\mapsto (y_{r+d_1+1},\dots, y_{r+d})$ implies the existence of a splitting $V\cong V'\times V''$, where $V'=\{(y_1,\dots, y_{r+d})\in V\mid y_{r+d_1+1},\dots, y_{r+d}=0\}$. The space $V'$ is naturally identified with $\Hom(\Pic(X,M),\mathbb{R})$, and the Haar measure $\nu'$ on $V'$ induced by $\nu$ on $V$ is the measure such that the volume of $V'/\Lambda'$ is $1$, where $\Lambda':=\Hom(\Pic(X,M),\mathbb{Z})$. Similarly $\nu$ induces the Haar measure $\nu''$ on $V''$ such that $V''/\Lambda''$ has volume equal to $1$, where $\Lambda''$ is the image of $\Lambda$ in $V''$. Under the isomorphism $V\cong V'\times V''$, the measure $\nu$ corresponds to the product measure $\nu'\times \nu''$, so Fubini's theorem implies that 

\begin{align*}
	I(B,\sigma)&=\frac{I(\sigma)}{\#\Pic(X,M(\sigma))_{\mathrm{torsion}}}\\ &\quad \times \int_{E_b\cap V'} \exp(y_1+\dots+y_{r+d_1}) \Volume(Z_\sigma(b-y_1-\dots-y_{r+d_1})) \d \nu' \\
	&= \frac{I(\sigma)\Volume(Z_\sigma(1))}{\#\Pic(X,M(\sigma))_{\mathrm{torsion}}} \\ &\quad \times \int_{E_b\cap V'} \exp(y_1+\dots+y_{r+d_1}) (b-y_1-\dots-y_{r+d_1})^{\dim (Z_\sigma(1))} \d \nu',
\end{align*}
where $Z_\sigma(c)=\{(y_{r+d_1+1},\dots,y_{r+d})\in V''\cap [0,\infty)^{d-d_1}\mid \sum_{i=r+d_1+1}^{r+d} \tilde{a}_i y_i\leq c\}$ is a polytope and the volume is with respect to the measure $\nu''$. Note that the polytope $Z_\sigma(1)$ is the polytope $Z_\sigma$ in Theorem \ref{theorem: full asymptotics}.

Under the identification of $V'$ with $\Hom(\Pic(X,M),\mathbb{R})$, the subset $E_b\cap V'$ is the set of linear forms $\varphi$ which are nonnegative on effective divisors and such that $\phi(-K_{(X,M)})\leq b$. Let $\lambda\colon V'\rightarrow \mathbb{R}$ be the linear form
$$\lambda(y_1,\dots,y_{r+d_1})=y_1+\dots+y_{r+d_1}$$ obtained by evaluating at the anticanonical class $-K_{(X,M)}$ of $(X,M)$. As the volume of the fibre above any $y\in [0,\infty)$ is equal to $\#\Pic(X,M)_{\mathrm{torsion}}\alpha_{\mathrm{Peyre}}((X,M),L)y^{b(\mathbb{Q},(X,M),L)-1}$, where $\alpha_{\mathrm{Peyre}}((X,M),L)$ is the variant of the $\alpha$-constant as given in \cite[Remark 8.6]{Moe25conjecture}, integrating along the fibres of $\lambda$ gives
\begin{align*}
	\int_{E_b\cap V'} &\exp(y_1+\dots+y_{r+d_1}) (b-y_1-\dots-y_{r+d_1})^{d-d_1} \d \nu'\\&=\#\Pic(X,M)_{\mathrm{torsion}}\alpha_{\mathrm{Peyre}}((X,M),L)\int_{0}^b \exp(y)y^{b(\mathbb{Q},(X,M),L)-1} (b-y)^{\dim (Z(1))}\d y.
\end{align*}

Thus we have shown the equality
\begin{align}
	I(B,\sigma)= &\frac{\# \Pic(X,M)_{\mathrm{torsion}}I(\sigma)\alpha_{\mathrm{Peyre}}((X,M),L)\Volume(Z(1))}{\# \Pic(X,M(\sigma))_{\mathrm{torsion}}}\\ & \times\int_{0}^b \exp(y)y^{b(\mathbb{Q},(X,M),L)-1} (b-y)^{\dim(Z_\sigma)}\d y.
\end{align}

We will now determine the main term in the integral.
\begin{proposition}
	For all $r,s\in \mathbb{N}$ and $b\in (1,\infty)$,
	$$\int_0^b \exp(y) y^r (b-y)^s\d y= s! \exp(b) b^{r}+O(\exp(b)b^{r-1})$$
	as $b\rightarrow \infty$.
\end{proposition}
\begin{proof}
	Let $$I(r,s,b):=\int_0^b \exp(y) y^r (b-y)^s\d y.$$
	If either $r$ or $s$ is zero, the integral is given by
	$$I(r,0,b)=\exp(b) \sum_{k=0}^r (-1)^k \frac{r!}{(r-k)!}b^{r-k}+(-1)^{r+1}r!,$$
	or
	$$I(0,s,b)=\exp(b)s!-\sum_{i=0}^s \frac{s!}{(s-i)!} b^{s-i}.$$
	Further integrating by parts, we obtain
	$I(r,s,b)=sI(r,s-1,b)-rI(r-1,s,b)$ if $r,s\geq 1$, which directly implies the result.
\end{proof}
Therefore we find
\begin{multline}
	I(B,\sigma)=\frac{\# \Pic(X,M)_{\mathrm{torsion}}I(\sigma)\alpha_{\mathrm{Peyre}}((X,M),L)\Volume(Z_\sigma)\dim(Z_\sigma)!}{\# \Pic(X,M(\sigma))_{\mathrm{torsion}}} \\ \times B(\log B)^{b(\mathbb{Q},(X,M),L)-1}+O(B(\log B)^{b(\mathbb{Q},(X,M),L)-2}).
\end{multline}
To finish the proof of Theorem \ref{theorem: full asymptotics} for toric adjoint rigid divisors, all that remains is to show that $\int_{D(B,\sigma)}\d \bfx\sim I(B,\sigma)$.
\begin{lemma} \label{lemma: reduction to I(B,sigma)}
	For every maximal cone $\sigma\in \Sigma_{\overline{M}}$,
	$$\int_{D(B,\sigma)} \d \bfx=I(B,\sigma)+O(B(\log B)^{b(\mathbb{Q},(X,M),L)-2})$$ as $B\rightarrow \infty$, so \eqref{eq: subdivide integral D(B) over cones} implies
	$$I(B)=\sum_{\sigma\in \Sigma_{\overline{M}}}I(B,\sigma)+O(B(\log B)^{b(\mathbb{Q},(X,M),L)-2})$$
	as $B\rightarrow \infty$.
\end{lemma}
\begin{proof}
	By \eqref{eq: full integral F(B)}, it suffices to show
	$$R_i=\int_{F(B)} \bfx^{-D_{(X,M)}}/x_i \frac{\d x_1}{x_1}\dots \frac{\d x_r}{x_r}= O(B(\log B)^{b(\mathbb{Q},(X,M),L)-2})$$
	as $B\rightarrow \infty$, for any $i=r+1,\dots, r+d_1$.
	
	Let $(X,M')$ be the pair given by $\fM'=\fM^\circ\setminus \{\bfm\}$, where $\bfm$ is the element corresponding to the coordinate $x_i$. Then $R_i=\int_{F(B)} \bfx^{-D_{(X,M')}} \frac{\d x_1}{x_1}\dots \frac{\d x_r}{x_r}$. We can estimate this integral in exactly the same way as in the computation of the asymptotic growth of $I(B,\sigma)$, to get
	$$R_i=O(B(\log B)^{b(\mathbb{Q},(X,M'),L)-1})$$
	as $B\rightarrow \infty$. Now since $[\tilde{D}_{i}]$ does not lie on the minimal face of $\Eff^1(X,M)$ containing $\pr^*_M L+K_{(X,M)}$, we have $b(\mathbb{Q},(X,M'),L)=b(\mathbb{Q},(X,M),L)-1$, which gives the result.
\end{proof}

Putting everything together, we conclude that the polynomial $Q$ satisfies
$$Q(\log B)=2^{\dim X} C_0 \tilde{C} (\log B)^{b(\mathbb{Q},(X,M),L)-1} +O((\log B)^{b(\mathbb{Q},(X,M),L)-2}),$$

where \begin{align*}
	C_0=&\prod_{\substack{p \, \mathrm{prime}\\ p\mid S}}(1-p^{-1})^{\#\Gamma_{M^\circ}} \sum_{\bfm\in \fM_{\red}}p^{-a_{\bfm}}\\
	\times &\prod_{\substack{p \, \mathrm{prime}\\ p\nmid S}}(1-p^{-1})^{\#\Gamma_{M^\circ}} \sum_{\bfm\in \mathbb{N}^n_{\red}}p^{-a_{\bfm}},
\end{align*}
and
$$\tilde{C}=\alpha_{\mathrm{Peyre}}((X,M),L)\sum_{\sigma \in \Sigma_{\overline{M},\max}} I(\sigma)C_\infty(\sigma),$$
where 
$$C_\infty(\sigma)=\frac{\Pic(X,M)_{\mathrm{torsion}}}{\Pic(X,M(\sigma))_{\mathrm{torsion}}}\Volume(Z_\sigma)\dim(Z_\sigma)!.$$
This finishes the proof of the toric adjoint rigid case of the theorem.

\textbf{The leading constant in the adjoint rigid case.}
From now on we will assume that $L$ is adjoint rigid with respect to $(X,M)$. As this implies that $L$ is toric adjoint rigid with respect to $(X,M)$, we can use the expression we derived for the leading coefficient of $Q$. We have already seen in Proposition \ref{prop: C0} that $\#\Gamma_{M^\circ}=\dim(X)+b(\mathbb{Q},(X,M),L)$. Therefore, all that remains is to prove
$$\tilde{C}=\alpha_{\mathrm{Peyre}}((X,M),L)C_{\infty},$$
where $C_{\infty}$ is the constant given in Theorem \ref{theorem: full asymptotics} in the adjoint rigid case.
We start by computing the volume of the simplex $Z_\sigma$.
\begin{proposition}
	The simplex $Z_\sigma$ has dimension $d-d_1$ and its volume is given by
	$$\Volume(Z_\sigma)=\frac{\Pic(X,M(\sigma))_{\mathrm{torsion}}}{(d-d_1)!\# \Pic(X,M)_{\mathrm{torsion}}\prod_{i=r+d_1+1}^{r+d} \tilde{a}_i},$$
	where we recall that $\tilde{a}_i$ is the coefficient of $\tilde{D}_i$ in $\pr^*_{M(\sigma)} L$.
\end{proposition}
\begin{proof}
	By Proposition \ref{prop: characterisation adjoint rigid}, the $\mathbb{Q}$-divisor class $\pr_{\overline{M}}^* L+K_{(X,\overline{M})}$ is rigid, so the subgroup $G$ of $\Pic(X,M(\sigma))$ generated by the divisors classes $[\tilde{D}_{r+d_1+1}],\dots, [\tilde{D}_{r+d}]$ is a free abelian group of rank $d-d_1$. Therefore $V''\cong \mathbb{R}^{d-d_1}$, and $Z_\sigma\subset V''$ is a simplex with vertices $$(\tilde{a}^{-1}_{r+d_1+1},0,\dots,0),\dots, (0,\dots,0,\tilde{a}^{-1}_{r+d}).$$ Thus the volume of $Z_\sigma$ is given by $$\Volume(Z_\sigma)=\frac{\Volume([0,1]^{d-d_1})}{(d-d_1)!\prod_{i=r+d_1+1}^{r+d} \tilde{a}_i}.$$ The volume of the hypercube $[0,1]^{d-d_1}$ with respect to the measure $\nu''$ is the reciprocal of the order of the kernel of the quotient homomorphism $\Pic(X,M(\sigma))/\{\mathrm{torsion}\}\rightarrow \Pic(X,M)/\{\mathrm{torsion}\}$ induced by the restriction $\Div(X,M(\sigma))\rightarrow \Div(X,M)$. Since the kernel $G$ of $\Pic(X,M(\sigma))\rightarrow \Pic(X,M)$ is torsion-free, the volume of the hypercube is $\Volume([0,1]^{d-d_1})=\frac{\# \Pic(X,M(\sigma))_{\mathrm{torsion}}}{\# \Pic(X,M)_{\mathrm{torsion}}}$.
\end{proof}
The previous proposition implies
$$C_\infty(\sigma)= \prod_{\substack{\bfm\in \Gamma_{\overline{M}}\\ \rho_\bfm\in \sigma}}\frac{1}{a_\bfm}$$
so
$$\tilde{C}=\alpha_{\mathrm{Peyre}}((X,M),L)\sum_{\sigma \in \Sigma_{\overline{M},\max}} I(\sigma)\prod_{\substack{\bfm\in \Gamma_{\overline{M}}\\ \rho_\bfm\in \sigma}}\frac{1}{a_\bfm}.$$
The following proposition finishes the proof of Theorem \ref{theorem: full asymptotics}.

\begin{proposition}
	For every maximal cone $\sigma'\in \Sigma$, we have
	$$\sum_{\sigma\subset \sigma'} I(\sigma)\prod_{\substack{\bfm\in \Gamma_{\overline{M}}\\ \rho_\bfm\in \sigma}}\frac{1}{a_\bfm}=\prod_{\substack{i=1\\ \rho_i\subset \sigma'}}^n \frac{1}{a_i},$$
	where the sum runs over all maximal cones in $\Sigma_{\overline{M}}$ contained in $\sigma'$.
\end{proposition}
\begin{proof}
	Let $\sigma$ be a maximal cone in $\Sigma_{\overline{M}}$.
	The index $I(\sigma)$ is equal to the cardinality of the quotient of $\Pic(X,M(\sigma))$ by the divisors $\{D_{\bfm}\mid \bfm \in \Gamma_{M(\sigma)}, \phi(\bfm)\not \in \sigma\}$. We can view this quotient as the Picard group of the pair $(X,M_\sigma)$, where $$\fM_\sigma=\{(0,\dots,0)\}\cup \{\bfm\in \Gamma_{M(\sigma)}\mid \phi(\bfm) \in \sigma\}.$$
	
	Now Proposition \ref{prop: Picard group toric pair} implies that $\Pic(X,M_\sigma)$ is the cokernel of the homomorphism $N^\vee\rightarrow \Div_T(X,M_\sigma)$. As $\sigma$ is a cone of dimension $d$, this homomorphism is an embedding of lattices of the same rank. Thus $\#\Pic(X,M_\sigma)$ is the index of $\Div_T(X,M_\sigma)^\vee$ in $N$, where the embedding is given by the dual of the homomorphism. As the image of a divisor $\tilde{D}_\bfm$ in $N$ is $\phi(\bfm)$, this implies that $\Pic(X,M_\sigma)$ has $|N:N_{M_\sigma}|$ elements, where we recall that $N_{M_{\sigma}}$ is the lattice spanned by $\{\phi(\bfm) \mid\bfm\in \fM_\sigma\}$. We choose a basis of $N\cong \mathbb{Z}^d$ so that $\sigma'=[0,\infty)^d$ and we write $\Gamma_{M_\sigma}=\{\bfm_1,\dots,\bfm_d\}$. The set $\{\phi(\bfm) \mid\bfm\in \Gamma_{M_\sigma}\}$ is a basis of $N_{\mathbb{Q}}$, so $I(\sigma)=|N:N_{M_\sigma}|$ is equal to the absolute value of the determinant $\phi(\bfm_1)\wedge\dots \wedge \phi(\bfm_d)$.
	
	We will prove the desired identity by viewing both sides as an exponential integral over a cone. We order the divisors $D_1,\dots, D_n$ on $X$ such that $\rho_1,\dots,\rho_d\in \sigma'$.
	Note that
	$$\prod_{i=1}^d \frac{1}{a_i}= \int_{[0,\infty)^d} e^{-a_1x_1-\dots-a_dx_d}\d x_1\dots\d x_d.$$
	We split up the domain of integration into the maximal cones $\sigma\in \Sigma_{\overline{M}}$ contained in $\sigma'=[0,\infty)^d$: 
	$$ \int_{[0,\infty)^d} e^{-a_1x_1-\dots-a_dx_d}\d x_1\dots\d x_d=\sum_{\sigma\subset \sigma'} \int_{\sigma} e^{-a_1x_1-\dots-a_dx_d}\d x_1\dots\d x_d.$$
	Now the formula \cite[Example 2.1]{Bar93} for the exponential integral over a cone gives $$\int_{\sigma} e^{-a_1x_1-\dots-a_dx_d}\d x_1\dots\d x_d=|\phi(\bfm_1)\wedge\dots \wedge \phi(\bfm_d)| \prod_{i=1}^d \frac{1}{\langle \mathbf{a}, \phi(\bfm_i)\rangle}$$
	where $\mathbf{a}=(a_1,\dots,a_n)$. Since $\langle \mathbf{a}, \phi(\bfm_i)\rangle=a_{\bfm_i}$, this implies
	$$\int_{\sigma} e^{-a_1x_1-\dots-a_dx_d}\d x_1\dots\d x_d=I(\sigma)\prod_{\substack{\bfm\in \Gamma_{\overline{M}}\\ \rho_\bfm\in \sigma}}\frac{1}{a_\bfm},$$
	which implies the desired identity.
\end{proof}
\printbibliography
\end{document}